\documentclass{amsart}

\usepackage{amsthm,amsmath,amsfonts,amssymb}
\usepackage{mathrsfs,latexsym,enumitem}
\usepackage{graphicx,caption,subcaption}
\usepackage{tkz-euclide}
\usetkzobj{all}

\newtheorem{theorem}{Theorem}[section]
\newtheorem{lemma}{Lemma}[section]
\newtheorem{proposition}{Proposition}[section]
\newtheorem{corollary}{Corollary}[section]

\numberwithin{equation}{section}

\begin{document}
\title[Projections of self similar measures]{Exact dimensionality and projections of random self-similar measures and sets}

\author{Kenneth Falconer}
\address{Mathematical Institute, University of St Andrews, North Haugh, St Andrews, Fife, KY16 9SS, Scotland}
\email{kjf@st-andrews.ac.uk}

\author{Xiong Jin}
\address{School of Mathematics, University of Manchester, Oxford Road, Manchester M13 9PL, United Kingdom}
\email{xiong.jin@manchester.ac.uk}
\thanks{The second author held a Royal Society Newton International Fellowship whilst this work was carried out.}

\begin{abstract}
We study the geometric properties of random multiplicative cascade measures defined on self-similar sets. We show that such measures and their projections and sections are almost surely exact-dimensional, generalizing Feng and Hu's result \cite{FeHu09} for self-similar measures. This, together with a compact group extension argument, enables us to generalize Hochman and Shmerkin's theorems on projections of deterministic self-similar measures  \cite{HoSh12} to these random measures without requiring any separation conditions on the underlying sets. We give applications to self-similar sets and fractal percolation, including new results on projections,  $C^1$-images and distance sets.
\end{abstract}

\maketitle
\section{Introduction}

Relating the Hausdorff dimension of a set $K \subseteq \mathbb{R}^d$ to the dimensions of its projections and sections has a long history. The most basic result, due to Marstrand \cite{Mar54} in the plane and to  Mattila \cite{Mat75} more generally, is that if $K\subseteq \mathbb{R}^d$ is Borel or analytic, then, writing $ \Pi_{d,k}$ for the  family of orthogonal projections from $\mathbb{R}^d$ onto its $k$-dimensional subspaces, 
\begin{equation}\label{a1}
\dim_H \pi K=\min(k,\dim_H K)
\end{equation}
for almost all  $\pi\in \Pi_{d,k}$ with respect to the natural invariant measure on $\Pi_{d,k}$, where $\dim_H $ denotes Hausdorff dimension. 
These papers also discuss the dimensions of sections or slices of sets and show that for almost all  $\pi\in \Pi_{d,k}$, if $ \dim_H K > k $, the sections $\pi^{-1}x \cap K$ satisfy
\begin{equation}\label{a2}
\dim_H (\pi^{-1}x \cap K)\leq \dim_H K - k 
\end{equation}
 for Lebesgue almost all $x \in \pi(K)$, with equality for a set of $x$ of positive Lebesgue measure; 
see \cite{Mat95}  for a good exposition of this material.

 The Hausdorff dimension of a probability measure $\mu$ is defined as 
 \begin{equation}\label{dimmes}
\dim_H \mu = \inf \{\dim_H K:\mu(K) >0\}. 
\end{equation}
The dimension properties of projections and sections of measures directly parallel those for sets; indeed the conclusions for sets generally follow from the measure analogues.

These classical results have been extended beyond recognition, for example to families of generalized projections \cite{PerSch00}, to obtain estimates on the size of `exceptional' projections $\pi$ for which the conclusions  (\ref{a1}) or (\ref{a2}) fail \cite{PerSch00}, and to packing dimensions  \cite{FalHow97}. Almost all of this work concerns sections and projections of general Borel or analytic sets $K$ for which the possibility of exceptional projections can never be excluded. Nevertheless, it has recently been noted that for specific classes of sets and measures the dimensions of projections or sections may be constant for all $\pi$, or at least  it may be possible to identify the exceptional $\pi$. In particular, highly innovative approaches of Hochman and Shmerkin \cite{HoSh12} and Furstenberg \cite{Fur} have addressed this for self-similar sets and measures, and our aim here is to generalise their results to a random setting. 

A family of contractions $\mathcal{I}=\{f_i\}_{i=1}^{m}$ on $\mathbb{R}^d$, referred to in this context as an {\it iterated function system} (IFS),  defines a unique non-empty compact set $K$ such that 
\begin{equation}\label{IFSatt}
K = \bigcup_{i=1}^m f_i(K);
\end{equation}
$K$ is termed the {\it attractor} of the IFS, see, for example, \cite{Falconer03}. Here we consider an IFS of contracting similarities
\begin{equation}\label{IFS}
\mathcal{I}=\{f_i=r_iO_i\cdot+t_i\}_{i=1}^{m},
\end{equation}
where each  $f_i$ is a composition of a scaling of ratio $r_i<1$, an orthonormal rotation $O_i$ and a translation $t_i$; we call such an attractor $K$ a {\it self-similar set}. Our conclusions will depend very much on the nature of  the {\it rotation group} $G$ of the IFS, that is the closure of the subgroup of $SO(d,\mathbb{R})$ generated by the $O_i$.

In this paper we obtain almost sure properties of projections and sections of {\it random multiplicative cascade measures} on self-similar sets.  The precise definition is given in Section 2.3 but for the purposes of this introduction such a measure will be denoted by $\widetilde{\mu}$ and be supported by a self-similar set $K$.  In particular,  $\widetilde{\mu}$ is {\it statistically self-similar}, that is, roughly speaking, the restriction  of $\widetilde{\mu}$ to each small scale component of $K$ has, after scaling, the same random distribution as  $\widetilde{\mu}$ itself.  Our motivation for considering such measures is that they are the natural random generalisations of self-similar measures but also are the natural tools for studying fractal percolation processes. Moreover, random cascade measures provide the classical models for multiplicative chaos theory, an area that has recently attracted attention because of  its connection to quantum gravity, see \cite{RV} for a recent survey. 

In Section 2 we give a precise construction of the probability space underlying the random cascade measures, and thus obtain an ergodic random dynamical system on the space of random cascade measures. An application of the compact group extension theorem shows that the skew product of this random dynamical system with the rotation group $G$ is also ergodic.

These ergodicities are used in Section 3 to show that almost surely a random multiplicative cascade measure $\widetilde{\mu}$, as well as almost all of its projections and sections (with respect to the Haar measure on $G$) are exact dimensional, that is the local dimensions exist and are constant almost everywhere. The proofs, which reformulate the measures of small balls as a type of Birkhoff sum, are adapted from the ergodic theoretic approach introduced for the deterministic case in \cite{FeHu09}. This sum converges to the  conditional entropy with respect to a sub-$\sigma$-algebra that captures the overlapping structure of self-similar sets, giving exact-dimensionality without any seperation condition (i.e. without requiring the union in (\ref{IFSatt}) to be disjoint), as well as a  formula for the exact dimension in terms of the conditional entropy.

One consequence of this is an almost sure `dimension conservation' property, relating the dimensions of the projections to those of perpendicular sections. Writing $\pi \widetilde{\mu}$ for the measure on $ \pi(K)$ obtained by projecting $\widetilde{\mu}$ under $\pi$, and $\widetilde{\mu}_{y,\pi}$ for the section of  $\widetilde{\mu}$ by the $(d-k)$-dimensional  plane $\pi^{-1}y$, we get the following conclusions when the rotation group is finite.

\begin{corollary}\label{t1}
Suppose  that $G$ is finite. Then for every projection  $\pi \in \Pi_{d,k}$,
\begin{equation}
\dim_H \pi \widetilde{\mu}+\dim_H \widetilde{\mu}_{y,\pi}=\dim_H \widetilde{\mu} 
\quad \mbox{ for $\pi \widetilde{\mu}$-almost all $y \in \pi(K)$}
\label{dimcorint}
\end{equation}  
almost surely. In particular, if $\widetilde{\mu}$ is deterministic then (\ref{dimcorint}) holds  for all $\pi$.
\end{corollary}
\begin{proof} See Corollaries \ref{dim-con} and \ref{dim-con2}.
\end{proof}

Note that  the deterministic case extends the result of Furstenberg \cite{Fur} by dispensing with the separation requirement that the union in (\ref{IFSatt}) is disjoint.

In Section 4 we show that if $G=SO(d,\mathbb{R})$ then almost surely all projections of  $\widetilde{\mu}$ and, indeed, all images  of  $\widetilde{\mu}$ under non-singular $C^1$-maps, have dimension equal to  the  `generic' value. The deterministic results that were proved using CP-processes in  \cite{HoSh12} follow as a special case. Here  we adopt a new approach  utilising the skew product dynamical system, leading to results such as the following.

\begin{corollary}\label{t2}
If  $G=SO(d,\mathbb{R})$ then almost surely, conditional on non-extinction of the random measure $\widetilde{\mu}$,
\begin{equation}
\dim_H \pi\widetilde{\mu} = \min(k, \dim_H\widetilde{\mu})  \quad  \text{ for all } \pi\in \Pi_{d,k}.\label{ineqbetathm1}
\end{equation}
More generally, for all $C^1$-maps $h:K\mapsto \mathbb{R}^k$ without singular points,
\begin{equation}
\dim_H h\widetilde{\mu}=\min(k,\dim_H \widetilde{\mu}).\label{eqnbthm1}
\end{equation}

\end{corollary}
\begin{proof} See Theorem \ref{e} and Corollary \ref{c2}.
\end{proof}

We specialise these results to deterministic self-similar sets in Section 5, and in particular show that conclusions relating to the dimensions of all projections and dimension conservation are valid without any separation condition on the self-similar construction, extending work of Hochman and Shmerkin \cite{HoSh12} and Furstenberg \cite{Fur}. Again there are consequences for the dimensions of images of sets under $C^1$-mappings and also for the dimensions of distances sets.

Recently there has been considerable interest in geometric properties of percolation on self-similar sets, that is random subsets $K_\mathbb{P}$ of $K$ obtained by removing components of the iterated construction of $K$ according to a  self-similar probability distribution $\mathbb{P}$. Associating the natural measures on   $K_\mathbb{P}$ with random cascade measures, we obtain in Section 6 new almost sure properties of projections and dimension conservation for  these random sets. 

\section{Preliminaries}

\subsection{Symbolic space}

Symbolic or code space underlies the structure of self-similar sets.

Let $\Lambda=\{1,\cdots,m\}$ be the alphabet on $m\ge 2$ symbols. Denote by $\Lambda^*=\cup_{n\ge 0} \Lambda^n$ the set of finite words, with the convention that $\Lambda^0=\{\emptyset\}$. Let $\Lambda^\mathbb{N}$ be the symbolic space of infinite sequences from the alphabet. For $\underline{i}\in \Lambda^\mathbb{N}$ and $n\ge 0$ let $\underline{i}|_n\in \Lambda^n$ be the first $n$ digits of $\underline{i}$. For $i\in \Lambda^n$ let $[i]=\{\underline{i}\in \Lambda: \underline{i}|_n=i\}$ be the {\it cylinder} rooted at $i$. We may endow $\Lambda^\mathbb{N}$ with the standard metric $d_\rho$ with respect to a number $\rho\in(0,1)$, that is for $\underline{i},\underline{j}\in \Lambda^\mathbb{N}$, $d_\rho(\underline{i},\underline{j})=\rho^{\inf\{n\ge 0: \underline{i}|_n\neq \underline{j}|_n\}}$. Then $(\Lambda^\mathbb{N},d_\rho)$ is a compact metric space. Let $\mathcal{B}$ be its Borel $\sigma$-algebra. Define the left-shift map $\sigma$ by  $\sigma(\underline{i})=(i_{n+1})_{n\ge 1}$ for $\underline{i}=(i_n)_{n\ge 1}\in \Lambda^\mathbb{N}$.

\subsection{Self-similar sets}

Let $\mathcal{I}$ be an IFS as in \eqref{IFS} with non-empty compact attractor $K\subseteq \mathbb{R}^d$ satisfying (\ref{IFSatt}).
For $i=i_1\cdots i_n\in \Lambda^n$ write
\[
f_i= f_{i_1}\circ \cdots \circ f_{i_n}=r_iO_i\cdot+t_i,
\]
where $r_i =  r_{i_1}\cdots r_{i_n}$, $O_i =  O_{i_1}\cdots O_{i_n}$ and $t_i$ is the appropriate translation.
Throughout the paper, $G=\overline{\langle O_i:i\in \Lambda\rangle}$ will denote the {\it rotation group} of the IFS, that is the compact subgroup of $SO(d,\mathbb{R})$ generated by the orthonormal maps $\{O_i,  i\in\Lambda\}$.

Let $\Phi: \Lambda^\mathbb{N} \mapsto K$ be the canonical projection, that is $\Phi(\underline{i})=\lim_{n\to\infty} f_{\underline{i}|_n}(x_0)$ for some $x_0\in K$. Let $R=\max\{|x|:x\in K\}$ and $\rho=\max\{r_i:i\in\Lambda\}$. Then it is easy to see that $\Phi: (\Lambda^\mathbb{N},d_\rho) \mapsto K$ is $R$-Lipschitz.

\subsection{Random multiplicative cascades}

A random multiplicative cascade is essentially a measure on $ \Lambda^\mathbb{N} $ constructed in a self-similar manner on the successive $ \Lambda^n$, see \cite{KaPe76,Ba99}.
Let $(\Omega,\mathcal{F},\mathbb{P})$ be a probability space. Let
\[
W=(W_i)_{i\in\Lambda}\in [0,\infty)^m
\]
be a random vector defined on $(\Omega,\mathcal{F},\mathbb{P})$ with $\sum_{i\in\Lambda}\mathbb{E}(W_i)=1$.
Let $\{W^{[i]}:i\in \Lambda^*\}$ be a sequence of independent and identically distributed random vectors having the same law as $W$. For $i\in\Lambda^*$, $n\ge 1$ and $j=j_1\cdots j_n\in\Lambda^n$ define
\[
Q^{[i]}_j=W_{j_1}^{[i]}W_{j_2}^{[ij_1]}\cdots W_{j_n}^{[ij_1\cdots j_{n-1}]},
\]
and for $i\in\Lambda^*$ and $n\ge 1$ define $Y_n^{[i]}=\sum_{j\in \Lambda^n} Q^{[i]}_j$. By definition $\{Y_n^{[i]}\}_{n\ge 1}$ is a non-negative martingale. Assume that
\begin{align}
{\bf(a0)} &\quad \mathbb{P}\left(\#\{i\in\Lambda: W_i>0\}>1\right)>0;\nonumber\\
{\bf(a1)} & \quad\mbox{There exists $p>1$ such that } \textstyle\sum_{i=1}^{m}\mathbb{E}\left(W_i^p\right)<1.
\label{assumptions}\end{align}
Then $Y_n^{[i]}$ converges a.s. to a nontrivial limit which we denote by $Y^{[i]}$, with expectation $\mathbb{E}(Y^{[i]}) = 1$. It is easy to see that $Y^{[i]}$, $i\in \Lambda^*$ have the same law as $Y=Y^{[\emptyset]}$. Moreover, for $p>1$ we have $\mathbb{E}(Y^p)<\infty$ if and only if  $\sum_{i=1}^{m}\mathbb{E}\left(W_i^p\right)<1$ (see \cite{DuLi83,KaPe76}). Since $\Lambda^*$ is countable, $Y^{[i]}$ is well-defined  for all $i\in\Lambda^*$ simultaneously. Moreover, by construction,
\begin{equation}\label{yi}
Y^{[i]}=\sum_{j=1}^{m} W_j^{[i]}Y^{[ij]}.
\end{equation}
Then for each $i\in\Lambda^*$ we may define a random measure $\mu^{[i]}$ on $\Lambda^\mathbb{N}$  by
\begin{equation}\label{mu}
\mu^{[i]}([j])=Q^{[i]}_j \cdot Y^{[ij]}, \quad \ j\in\Lambda^*.
\end{equation}
The measure $\mu^{[i]}$ is called the {\it random multiplicative cascade measure} generated by the sequence $\{W^{[ij]}:j\in \Lambda^*\}$. By definition the sequence $\{\mu^{[i]}: i\in\Lambda^*\}$ has the same law. Moreover, by  \eqref{yi} we have statistical self-similarity in the sense that for $i\in\Lambda^*$ and $j\in\Lambda^n$,
\begin{equation}\label{ss}
\mu^{[i]}\big|_{[j]}=Q^{[i]}_j\cdot \mu^{[ij]}\circ \sigma^{-n}\big|_{[j]}.
\end{equation}
Sometimes we will write $(\cdot)=(\cdot)^{[\emptyset]}$, in particular $Q_j=Q_i^{[\emptyset]}$ and $\mu=\mu^{[\emptyset]}$. Our main interest will be in random cascade measures on the self-similar set $K$ given by the canonical projection $\Phi \mu$ of  $\mu$ onto $K$. For more on random cascade measures, see \cite{BaJi10} and the references therein.

\subsection{The underlying probability space}

We now give a precise definition of the probability space on which the i.i.d. sequence $\{W^{[i]}:i\in\Lambda^*\}$ is defined. First recall that the random vector $W$ is defined on the probability space $(\Omega,\mathcal{F},\mathbb{P})$. We will work on the countable product space
\[
(\Omega^*,\mathcal{F}^*,\mathbb{P}^*)={\textstyle \bigotimes}_{i\in\Lambda^*}(\Omega_i,\mathcal{F}_i,\mathbb{P}_i),
\]
where $(\Omega_i,\mathcal{F}_i,\mathbb{P}_i)=(\Omega,\mathcal{F},\mathbb{P})$ for each $i\in\Lambda^*$. For $i\in\Lambda^*$ define the projection
\[
\pi_i:\Omega^{*}\mapsto \Omega_i.
\]
Then by letting $W^{[i]}=W\circ \pi_i$ for $i\in\Lambda^*$ we obtain a family of i.i.d. random vectors on $(\Omega^*,\mathcal{F}^*,\mathbb{P}^*)$. For $i\in\Lambda^*$ let $\mu^{[i]}\equiv\mu^{[i]}(\cdot,\omega)$ be the random cascade measure generated by the sequence $\{W^{[ij]}:j\in \Lambda^*\}$, as in \eqref{mu}. For $i\in\Lambda^*$ define
\[
\eta_i:\Omega^{*} \ni (\omega_j)_{j\in \Lambda^*} \mapsto (\omega_{ij})_{j\in\Lambda^*}\in \Omega^*.
\]
By definition $W^{[ij]}=W^{[i]}\circ \eta_j$ for all $i,j\in \Lambda^*$, thus
\begin{equation}\label{muij}
\mu^{[ij]}(\cdot,\omega)=\mu^{[i]}(\cdot,\eta_j\omega).
\end{equation}
Consequently, from \eqref{ss}, for any $B\in \mathcal{B}$,
\begin{eqnarray*}
\mu^{[i]}(B\cap [j],\omega)&=&Q^{[i]}_j(\omega)\cdot \mu^{[ij]}(\sigma^{-n}(B\cap[j]),\omega)\\
&=&Q^{[i]}_j(\omega)\cdot \mu^{[i]}(\sigma^{-n}(B\cap[j]),\eta_j\omega).
\end{eqnarray*}

\subsection{The Peyri\`ere measure}

Let $(\Omega',\mathcal{F}')=( \Lambda^\mathbb{N}\times \Omega^*, \mathcal{B}\otimes \mathcal{F}^*)$.  Let $\mathbb{Q}$ be the {\it Peyri\`ere measure} on $(\Omega',\mathcal{F}')$ with respect to $\mu=\mu^{[\emptyset]}$, that is for all $A\in\mathcal{F}'$,
\begin{equation}
\mathbb{Q}(A)=\int_{\Omega^*}\int_{\Lambda^\mathbb{N}} \chi_A(\underline{i},\omega) \, \mu(\mathrm{d}\underline{i},\omega)\, \mathbb{P}^*(\mathrm{d}\omega).\label{peydef}
\end{equation}
It is easy to see that $(\Omega',\mathcal{F}',\mathbb{Q})$ is a probability space. Notice that the inside integral is only defined when $\mu$ is not trivial. Write  $\mathbb{P}_*(A)=\mathbb{P}^*(A\cap \{\|\mu\|> 0\})/\mathbb{P}^*(\{\|\mu\|>0\})$ for $A\in\mathcal{F}^*$ for the probability conditional on $\mu$ being non-trivial. Thus ``for $\mathbb{Q}$-a.e. $(\underline{i},\omega)$" is equivalent to ``for $\mathbb{P}_*$-almost all $\mu$, and $\mu$-a.e. $\underline{i}$". Define the skew product
\[
T:\Omega'\ni (\underline{i},\omega)\mapsto (\sigma\underline{i},\eta_{\underline{i}|_1}(\omega)) \in \Omega'.
\]

\begin{lemma}\label{Ti}
The Peyri\`ere measure $\mathbb{Q}$ is $T$-invariant.
\end{lemma}

\begin{proof}
For all $B\in\mathcal{F}'$
\begin{eqnarray*}
\mathbb{Q}(T^{-1}B)&=&\int_{\Omega^*}\int_{\Lambda^\mathbb{N}} \chi_{T^{-1}B}(\underline{i},\omega) \, \mu(\mathrm{d}\underline{i},\omega)\, \mathbb{P}^*(\mathrm{d}\omega)\\
&=&\int_{\Omega^*}\int_{\Lambda^\mathbb{N}} \chi_{B}(\sigma\underline{i},\eta_{\underline{i}|_1}\omega) \, \mu(\mathrm{d}\underline{i},\omega)\, \mathbb{P}^*(\mathrm{d}\omega)\\
&=&\sum_{j\in\Lambda}\int_{\Omega^*}\int_{[j]} \chi_{B}(\sigma\underline{i},\eta_{j}\omega) \, \mu(\mathrm{d}\underline{i},\omega)\, \mathbb{P}^*(\mathrm{d}\omega)\\
&=&\sum_{j\in\Lambda}\int_{\Omega^*}W^{[\emptyset]}_j(\omega)\int_{[j]} \chi_{B}(\sigma\underline{i},\eta_{j}\omega) \, \mu(\mathrm{d}\sigma\underline{i},\eta_j\omega)\, \mathbb{P}^*(\mathrm{d}\omega)\\
&=&\sum_{j\in\Lambda}\int_{\Omega^*}W^{[\emptyset]}_j(\omega)\int_{\Lambda^{\mathbb{N}}} \chi_{B}(\underline{i},\eta_{j}\omega) \, \mu(\mathrm{d}\underline{i},\eta_j\omega)\, \mathbb{P}^*(\mathrm{d}\omega)\\
&=&\sum_{j\in\Lambda}\mathbb{E}(W_j)\mathbb{Q}(B)\\
&=&\mathbb{Q}(B).
\end{eqnarray*}
\end{proof}

\begin{proposition}\label{ge}
The dynamical system $(\Omega',\mathcal{F}',\mathbb{Q},T)$ is mixing.
\end{proposition}
\begin{proof}
Let $\mathcal{A}$ be the semi-algebra consisting of sets of the form
\[
\underline{i}|_k=j, \ W^b_a\in B^b_a, \ a\in \Lambda, \ b\in \cup_{i=1}^k \Lambda^i,
\]
for $k\in\mathbb{N}$, $j\in \Lambda^k$ and $B^b_a$ Borel subsets of $[0,\infty)$. It is clear that $\mathcal{A}$ generates $\mathcal{F}'$, so we only need to verify that for $A,B\in \mathcal{A}$, $\lim_{n\to\infty}\mathbb{Q}(T^{-n}A\cap B)=\mathbb{Q}(A)\mathbb{Q}(B)$. This follows from the fact that by the construction of $\mathcal{A}$, given $A,B\in\mathcal{A}$, there exists $n_0$ such that $T^{-n}A$ and $B$ are independent for all $n\ge n_0$. 
\end{proof}

\subsection{Normalised random cascade measures}

For $i\in\Lambda^*$ define
\[
\bar{\mu}_{[i]}=\chi_{\{\mu([i])>0\}}\frac{\mu|_{[i]}}{\mu([i])} \quad \text{ and }  \quad \bar{\mu}^{[i]}=\chi_{\{\|\mu^{[i]}\|>0\}}\frac{\mu^{[i]}}{\|\mu^{[i]}\|},
\]
with the convention that $\bar{\mu}=\bar{\mu}^{[\emptyset]}$. Then $\bar{\mu}_{[i]}$ and  $ \bar{\mu}^{[i]}$ are either probability measures or trivial. If $|i|=n$, then from \eqref{ss} we have
\begin{equation}\label{rescal}
 \bar{\mu}_{[i]}\circ \sigma^{-n}=\chi_{\{Q_i>0\}}\bar{\mu}^{[i]}.
\end{equation}
The measure sequence $\{\bar{\mu}^{[\cdot|_n]}\}_{n\ge 0}$ is a stationary process under the Peyri\`ere measure. This sequence is similar to Furstenberg's CP-processes: Let $\Delta$ be the natural partition operator on symbolic space: $\Delta [i]=\{[ij]:j\in \Lambda\}$ for $i\in\Lambda^*$. Starting from $(\bar{\mu},[\emptyset])$ we move to  $(\bar{\mu}^{[i]},[i])$ with probability $\bar{\mu}([i])$ for $i\in \Lambda$, and from $(\bar{\mu}^{[i]},[i])$ we move to $(\bar{\mu}^{[ij]},[ij])$ with probability $\bar{\mu}^{[i]}([j])$ for $j\in\Lambda$, and continue in this way. The resulting measure sequence clearly falls into the same sample space as $\{\bar{\mu}^{[\cdot|_n]}\}_{n\ge 0}$, but it seems unlikely they will have the same law unless the random cascade measures degenerate to Bernoulli measures.

\subsection{The compact group extension}

Let $G=\overline{\langle O_i:i\in \Lambda\rangle}$ be the closed subgroup of $SO(d,\mathbb{R})$ generated by the orthogonal maps $\{O_i,  i\in\Lambda\}$. For future reference note that $G$ also equals the closed subsemigroup  generated by the orthogonal maps $\{O_i,  i\in\Lambda\}$; this follows since the inverse of any element in a compact group can be approximated arbitrarily closely by positive powers of the element.  Let $\mathcal{B}_G$ be  Borel $\sigma$-algebra of $G$ and let $\xi$ be its normalized Haar measure. Define the measurable map $\phi: \Omega' \ni (\underline{i},\omega)\mapsto  O_{\underline{i}|_1}\in G$. Let $X=\Omega'\times G$ and define the skew product
\[
T_\phi: X \ni  (\omega',g) \mapsto (T\omega',  g\phi(\omega')) \in X.
\]
It is easy to verify that the product measure $\mathbb{Q}\times \xi$ is $T_\phi$-invariant.

\begin{proposition}\label{gee}
The dynamical system $(X,\mathcal{F}'\otimes \mathcal{B}_G,\mathbb{Q}\times \xi,T_\phi )$ is ergodic.
\end{proposition}

\begin{proof}
From Proposition \ref{ge} we know that $(\Omega',\mathcal{F}',\mathbb{Q},T)$ is ergodic. Using the compact group extension theorem, see for example \cite{KeyNew76}, $T_\phi$ is ergodic if and only if the equation
\begin{equation}\label{ir}
F(T\omega')=R(\phi(\omega'))F(\omega') \text{ for } \mathbb{Q}\text{-a.e. } \omega',
\end{equation}
where $R$ is an irreducible (unitary) representation (of degree $k$, say) and $F:\Omega'\mapsto \mathbb{C}^k$ is measurable, has only the trivial solution $R$, the trivial $1$-dimensional representation, with $F$ constant. Let $\mu_p$ is the Bernoulli measure on $\Lambda^\mathbb{N}$ corresponding to the probability vector $p=(\mathbb{E}(W_i))_{i\in\Lambda}$. From the measuable function $F$ in \eqref{ir} we may construct the following vector measure $\lambda$ on $\Lambda^\mathbb{N}$, defined as 
\[
\lambda(I)=\int_{\Omega^*}\int_I F(\underline{i},\omega)\, \mu(\mathrm{d}\underline{i},\omega) \, \mathbb{P}^*(\mathrm{d}\omega), \ \forall\, I \in\mathcal{B}.
\]
Then $\lambda$ is absolutely continuous with respect to $\mu_p$ since, for any set $E\in \mathcal{B}$ with $\mu_p(E)=0$, 
\begin{eqnarray*}
|\lambda(E)|&\le& \limsup_{R\to \infty} \int_{\Omega^*}\int_{E} \chi_{\{|F|\le R\}} |F(\underline{i},\omega)|\, \mu(\mathrm{d}\underline{i},\omega) \, \mathbb{P}^*(\mathrm{d}\omega)\\
&\le&  \limsup_{R\to \infty} R\cdot \mu_p(E)=0.
\end{eqnarray*}\label{rnd}
Denote by $f=\mathrm{d}\lambda /\mathrm{d}\mu_p$ the corresponding Radon-Nikodym derivative. In particular 
\begin{equation}\label{rnd}
f(\underline{i})=\lim_{n\to\infty}\frac{\lambda([\underline{i}|_n])}{\mu_p([\underline{i}|_n])} \text{ for } \mu_p\text{-a.e. } \underline{i}.
\end{equation}
Now fix $I=[i_1i_2\cdots i_n]$. From \eqref{ir} 
\begin{eqnarray*}
R(O_{i_1})\lambda([i_1i_2\cdots i_n])&=&\int_{\Omega^*}\int_{[i_1i_2\cdots i_n]} R(O_{i_1}) F(\underline{i},\omega)\, \mu(\mathrm{d}\underline{i},\omega) \, \mathbb{P}^*(\mathrm{d}\omega)\\
&=&\int_{\Omega^*}\int_{[i_1i_2\cdots i_n]} F(\sigma\underline{i},\eta_{i_1}\omega)\, \mu(\mathrm{d}\underline{i},\omega) \, \mathbb{P}^*(\mathrm{d}\omega)\\
&=&\int_{\Omega^*}W_{i_1}\int_{[i_2\cdots i_n]} F(\underline{i},\omega)\, \mu^{[i_1]}(\mathrm{d}\underline{i},\omega) \, \mathbb{P}^*(\mathrm{d}\omega)\\
&=&\mu_p([i_1])\lambda([i_2\cdots i_n]).
\end{eqnarray*}
This yields
\[
\frac{\lambda([i_2\cdots i_n])}{\mu_p([i_2\cdots i_n])}=R(O_{i_1})\frac{\lambda([i_1i_2\cdots i_n])}{\mu_p([i_1i_2\cdots i_n])}.
\]
Together with \eqref{rnd} we finally get
\[
f(\sigma \underline{i})=R(O_{\underline{i}|_1}) f (\underline{i}) \text{ for } \mu_p\text{-a.e. } \underline{i}.
\]
From \cite[Corollary 4.5]{Pa97} we know that the dynamical system $(\Lambda^\mathbb{N}\times G,\mathcal{B}\otimes \mathcal{B}_G,\mu_p\times \xi,\sigma_\phi)$ is ergodic, where $\sigma_\phi(\underline{i},g)=(\sigma \underline{i}, gO_{\underline{i}|_1})$ is a compact group extension of the Bernoulli full-shift with $\sigma_\phi$ having a dense orbit. By using the compact group extension theorem again this implies that $R$ must be the trivial $1$-dimensional representation. Applying this to \eqref{ir} we get that
\[
F(T\omega')=F(\omega') \text{ for } \mathbb{Q}\text{-a.e. } \omega',
\]
so  $F$ is constant using Proposition \ref{ge}.
\end{proof}

\subsection{Dimension and entropy}\label{secde} Let $\varphi:Y\mapsto Z$ be a continuous mapping between two metric spaces $Y$ and $Z$. For a Borel measure $\nu$ on $Y$ write
\[
\varphi\nu=\nu\circ \varphi^{-1}
\]
for the pull-back measure of $\nu$ on $Z$ through $\varphi$.

For a measure $\nu$ and $x\in \mathrm{supp}(\nu)$ let
\[
D_\nu(x)=\lim_{r\to 0}\frac{\log \nu(B(x,r))}{\log r}
\]
whenever the limit exists, where $B(x,r)$ is the closed ball of centre $x$ and radius $r$. If for some $\alpha\ge 0$ we have $D_\nu(x)=\alpha$ for $\nu$-a.e. $x$ we say that $\nu$ is {\it exact-dimensional}.

For $0<r<1$ and $\nu$ a probability measure supported by a compact subset $A$ of $\mathbb{R}^d$, let
\[
H_r(\nu)=-\int_A \log \nu(B(x,r)) \, \nu(\mathrm{d}x)
\]
be the $r$-{\it scaling entropy} of $\nu$. Note that, writing ${\mathcal M}$ for the probability measures supported by $A$, the map $H_r: {\mathcal M} \to \mathbb{R}\cup\{\infty\}$ need not be continuous in the weak-$\star$ topology. However, $H_r$ is lower semicontinuous as it may be expressed as the limit of an increasing sequence of continuous functions of the form $\nu \mapsto \int \max\{k, \log(1/\int f_k (x-y)\nu(dy))\nu(dx)\}$ where $f_k$ is a decreasing sequence of continuous functions approximating $\chi_{B(0,r)}$.  
The {\it lower entropy dimension} of $\nu$ is defined as
\[
\dim_e \nu=\liminf_{r\to 0}\frac{H_r(\nu)}{-\log r}
\]
and the {\it Hausdorff dimension}  of $\nu$ is $\dim_H \nu = \inf\{\dim_H A : \nu(A) >0\}$.
Then
\[
\dim_H \nu\le \dim_e \nu,
\]
with equality  when $\nu$ is exact-dimensional, see \cite{Fal97,Falconer03}.

\subsection{Conditional measure, information and entropy}

The following result is the conditional measure theorem of Rohlin \cite{Ro49} adapted to symbolic spaces.

\begin{theorem}\label{Rohlin}
Let $\eta$ be a countable $\mathcal{B}$-measurable partition of $\Lambda^\mathbb{N}$ in the sense that the quotient space $\Lambda^\mathbb{N}/\eta$ is separated by a countable number of measurable sets in $\mathcal{B}$. Let $\nu$ be a Borel probability measure on $\Lambda^\mathbb{N}$. Then for every $\underline{i}$ in a set of full $\nu$-measure, there is a probability measure $\nu_{\underline{i}}^\eta$ defined on $\eta(\underline{i})$ (the unique element in $\eta$ that contains $\underline{i}$) such that for any measurable set $B\in\mathcal{B}$, the mapping $\underline{i}\mapsto \nu_{\underline{i}}^\eta(B)$ is $\widehat{\eta}$-measurable ($\widehat{\eta}$ is the $\sigma$-algebra generated by $\eta$) and
\[
\nu(B)=\int_{\Lambda^{\mathbb{N}}} \nu_{\underline{i}}^\eta(B) \, \nu(\mathrm{d}\underline{i}).
\]
These properties imply that for any $f\in L^1(\Lambda^\mathbb{N},\mathcal{B},\nu)$ we have $\nu_{\underline{i}}^\eta(f)=\mathbb{E}_\nu(f|\widehat{\eta})$ for $\nu$-a.e. $\underline{i}$, and $\nu(f)=\int \mathbb{E}_\nu(f|\widehat{\eta}) \, \mathrm{d}\nu$.
\end{theorem}

For any sub-Borel $\sigma$-algebra $\mathcal{A}$ of $\mathcal{B}$, any countable $\mathcal{B}$-measurable partition $\mathcal{P}$ of $\Lambda^{\mathbb{N}}$, and any Borel probability measure $\nu$ on $\Lambda^{\mathbb{N}}$ we define the {\it conditional information}
\[
\mathbf{I}_\nu(\mathcal{P}\, |\, \mathcal{A})=-\sum_{B\in\mathcal{P}} \chi_B\log\mathbb{E}_\nu(\chi_B \, |\, \mathcal{A} )
\]
and the {\it conditional entropy}
\[
\mathbf{H}_\nu(\mathcal{P}\, |\, \mathcal{A})=\int_{\Lambda^\mathbb{N}} \mathbf{I}_\nu(\mathcal{P}\, |\, \mathcal{A})(\underline{i}) \, \nu(\mathrm{d}\underline{i}).
\]
For the trivial $\sigma$-algebra $\mathcal{N}=\{\emptyset, \Lambda^{\mathbb{N}}\}$ we use the convention that $\mathbf{I}_\nu(\mathcal{P})=\mathbf{I}_\nu(\mathcal{P}\, |\, \mathcal{N})$ and $\mathbf{H}_\nu(\mathcal{P})=\mathbf{H}_\nu(\mathcal{P}\, |\, \mathcal{N})$.

\medskip

We  state, in our notation and context, the following result from Feng \& Hu \cite{FeHu09} which we will need in several places.

\begin{proposition}\label{fenghu}
Let $\nu$ be a Borel probability measure on $\Lambda^{\mathbb{N}}$. Let $\eta$ and $\mathcal{P}$ be two countable measurable partitions of $\Lambda^{\mathbb{N}}$. Let $\varphi:\Lambda^{\mathbb{N}}\mapsto \mathbb{R}^d$ be a continuous function and denote by $\mathcal{B}_\varphi$ the $\sigma$-algebra generated by $\varphi^{-1}\mathcal{B}(\mathbb{R}^d)$. Then for $\nu$-a.e. $\underline{i}\in\Lambda^{\mathbb{N}}$,
$$\lim_{r \to 0} \log \frac{\nu^\eta_{\underline{i}}\left(\varphi^{-1}(B(\varphi(\underline{i}),r))\cap \mathcal{P}(\underline{i})\right)}{\nu^\eta_{\underline{i}}\left(\varphi^{-1}(B(\varphi(\underline{i}),r))\right)} = -\mathbf{I}_{m}(\mathcal{P} \,|\, \widehat{\eta}\vee\mathcal{B}_\varphi)(\underline{i}).$$
Moreover, writing
$$h(\underline{i}) = -\inf_{r>0} \log \frac{\nu^\eta_{\underline{i}}\left(\varphi^{-1}(B(\varphi(\underline{i}),r))\cap \mathcal{P}(\underline{i})\right)}{\nu^\eta_{\underline{i}}\left(\varphi^{-1}(B(\varphi(\underline{i}),r))\right)}$$
and assuming $\mathbf{H}_{\nu}(\mathcal{P})<\infty$, then $h \geq 0$ and $h\in L^1(\Lambda^{\mathbb{N}})$ with 
$\int_{\Lambda^{\mathbb{N}}} h(\underline{i}) \leq \mathbf{H}_{\nu}(\mathcal{P})+C_d$, where $C_d$ depends only on  $d$.
\end{proposition}

\begin{proof}
This is proved in \cite[Proposition 3.5]{FeHu09}. The bound for $\int_{\Lambda^{\mathbb{N}}} h(\underline{i})$ is contained within the proof.
\end{proof}

\section{Exact-dimensionality}
\label{seced}

In this section we establish the exact-dimensionality of random cascade measures on self-similar sets without any separation condition, as well as of the projections of the measures onto subspaces and of sliced measures for $\xi$ almost all rotations.

Let $\pi\in\Pi_{d,k}$. For $\underline{i} \in  \Lambda^{\mathbb{N}}$ define the {\it fibre}   
\[
[\underline{i}]_{\pi}=(\pi\Phi)^{-1}(\pi\Phi(\underline{i})),
\]
and write $\mathcal{P}_\pi=\{[\underline{i}]_\pi:\underline{i} \in \Lambda^{\mathbb{N}}\}$. It is a measurable partition since the quotient space $ \Lambda^{\mathbb{N}}/\mathcal{P}_\pi $ is separated by $\{(\pi\Phi)^{-1}B_i\}$ where $\{B_i\}$ is the sequence of closed cubes in $\pi(\mathbb{R}^d)$ with rational vertices. Denote by $\widehat{\mathcal{P}_\pi}$ the $\sigma$-algebra generated by $\mathcal{P}_\pi$. Due to Theorem \ref{Rohlin}, given the measurable partition $\mathcal{P}_\pi$, for any probability measure $\nu$ on $(\Lambda^{\mathbb{N}},\mathcal{B})$, for every $\underline{i}$ in a set of full $\nu$-measure, there is a probability measure $\nu^{\mathcal{P}_\pi}_{\underline{i}}$, which we shortly denote by $\nu_{\underline{i},\pi}$, defined on $\mathcal{P}_\pi(\underline{i})=[\underline{i}]_\pi$ such that for any $B\in\mathcal{B}$, the mapping $\underline{i}\mapsto \nu_{\underline{i},\pi}(B)$ is $\widehat{\mathcal{P}_\pi}$-measurable and
\[
\nu(B)=\int_{\Lambda^{\mathbb{N}}} \nu_{\underline{i},\pi}(B) \, \nu(\mathrm{d}\underline{i}).
\]
Furthermore for any $f\in L^1(\Lambda^\mathbb{N},\mathcal{B},\nu)$ we have $\nu_{\underline{i},\pi}(f)=\mathbb{E}_\nu(f|\widehat{\mathcal{P}_\pi})$ for $\nu$-a.e. $\underline{i}$, and $\nu(f)=\int \mathbb{E}_\nu(f|\widehat{\mathcal{P}_\pi}) \, \mathrm{d}\nu$. Moreover, $\nu_{\underline{i},\pi}$  depends only on $[\underline{i}]_\pi$; thus we may write $\nu_{y,\pi}=\nu_{\underline{i},\pi}$ when $y\in \pi(K)$ is such that $\pi\Phi(\underline{i})=y$. By definition for every Borel set $A\in\mathcal{B}$
\begin{equation}
\nu(A)= \int_{\Lambda^{\mathbb{N}}} \nu_{\underline{i},\pi}(A) \, \nu(\mathrm{d}\underline{i})=\int_{y\in \pi (K)} \nu_{y,\pi}(A)\pi \Phi\nu(\mathrm{d}y).\label{disinter}
\end{equation}

The following lemma, which is a variant of properties stated in \cite[Chapter 10]{Mat95}, expresses these conditional measures geometrically as limits of measures of narrow slices.   

\begin{lemma}\label{condimea}
For every set $A\in \mathcal{B}$, for $\pi\Phi\nu$-a.e. $y\in \pi(K)$, 
\[
\nu_{y,\pi}(A)=\lim_{r\to 0} \frac{\nu\left(A\cap \Phi^{-1}\pi^{-1} \left(B(y,r)\right)\right)}{\nu\left(\Phi^{-1} \pi^{-1} \left(B(y,r)\right)\right)},
\]
or equivalently for $\nu$-a.e. $\underline{i}\in\Lambda^\mathbb{N}$,
\[
\nu_{\underline{i},\pi}(A)=\lim_{r\to 0} \frac{\nu\left(A\cap \Phi^{-1}\pi^{-1} \left(B(\pi\Phi(\underline{i}),r)\right)\right)}{\nu\left(\Phi^{-1} \pi^{-1} \left(B(\pi\Phi(\underline{i}),r)\right)\right)}.
\]
\end{lemma}

\begin{proof}
Let $f:\underline{i}\mapsto \nu_{\underline{i},\pi}(A)$ and $\bar{f}:y\mapsto \nu_{y,\pi}(A)$. By \eqref{disinter},   for any $B\in\mathcal{B}(\pi(\mathbb{R}^d))$,
\begin{equation}\label{dis2}
\nu(A\cap \Phi^{-1}\pi^{-1} B)=\int_{\Phi^{-1}\pi^{-1}B} f \, \mathrm{d}\nu=\int_{B}\bar{f} \, \mathrm{d}\pi\Phi\nu.
\end{equation}
Define a measure $\lambda$ on $\pi(K)$ by $\lambda(B)=\nu(A\cap \Phi^{-1}\pi^{-1} B)$ for $B\in\mathcal{B}(\pi(\mathbb{R}^d))$. By \eqref{dis2} $\lambda$ is absolutely continuous with respect to $\pi\Phi\nu$ with
\[
\frac{\lambda\left(B(y,r)\right)}{\pi\Phi\nu\left(B(y,r)\right)}=\frac{1}{\pi\Phi\nu\left(B(y,r)\right)}\int_{B(y,r)} \bar{f} \, \mathrm{d}\pi\Phi\nu.
\]
Letting $r \to 0$ and applying the differentiation theory of measures, see for example, \cite[Theorem 2.12]{Mat95}, 
\[
\lim_{r\to 0}\frac{\nu\left(A\cap \Phi^{-1}\pi^{-1} \left(B(y,r)\right)\right)}{\nu\left(\Phi^{-1}\pi^{-1} B(y,r)\right)} =\nu_{y,\pi}(A)
\]
for $\pi\Phi\nu$-a.e. $y$, as required.
\end{proof}

 Let $\pi\in \Pi_{d,k}$ be fixed. Here is the main theorem of this section.

\begin{theorem}\label{thmed}
$\mathbb{P}_*$-a.s.,
\begin{itemize}
\item[(i)] $\Phi\mu$ is exact-dimensional with dimension
\[
\alpha=\frac{\mathbb{E}(\mathbf{H}_{\bar{\mu}}(\mathcal{P} \, |\, \mathcal{B}_\Phi))+\sum_{i=1}^m\mathbb{E}(W_i\log W_i)}{\sum_{i=1}^m\mathbb{E}(W_i)\log r_i}.
\]

\item[(ii)]  For $\xi$-a.e. $g\in G$, $\pi g\Phi\mu$ is exact-dimensional with dimension
\[
\beta(\pi)=\frac{\mathbb{E}_{\mathbb{P}^*\times \xi}(\mathbf{H}_{\bar{\mu}}(\mathcal{P} \, |\, \mathcal{B}_{\pi g\Phi}))+\sum_{i=1}^m\mathbb{E}(W_i\log W_i)}{\sum_{i=1}^m\mathbb{E}(W_i)\log r_i}.
\]
\item[(iii)]  For $\xi$-a.e. $g\in G$, for $\pi g\Phi\mu$-a.e. $y\in\pi g(K)$, $\Phi\bar{\mu}_{y,\pi g}$ is exact-dimensional with dimension
\[
\gamma(\pi)=\frac{\mathbb{E}(\mathbf{H}_{\bar{\mu}}(\mathcal{P} \, |\, \mathcal{B}_{\Phi}))-\mathbb{E}_{\mathbb{P}^*\times \xi}(\mathbf{H}_{\bar{\mu}}(\mathcal{P} \, |\,  \mathcal{B}_{\pi g\Phi}))}{\sum_{i=1}^m\mathbb{E}(W_i)\log r_i}.
\]
\end{itemize}
\end{theorem}

`Dimension conservation' for $\xi$-almost all rotations now follows. 
\begin{corollary}\label{dim-con}
$\mathbb{P}_*$-a.s. for $\xi$-a.e. $g\in G$ and $\pi g\Phi\mu$-a.e. $y\in\pi g(K)$,
\[
\dim_H \pi g\Phi\bar{\mu}+\dim_H \Phi\bar{\mu}_{y,\pi g}=\dim_H \Phi\bar{\mu}.
\]
\end{corollary}

\begin{proof}
If follows from Theorem \ref{thmed} that these measures are exact-dimensional and $\alpha=\beta(\pi)+\gamma(\pi)$.
\end{proof}

We immediately get the following corollary.

\begin{corollary}\label{dim-con2}
If $G$ is finite then for every projection $\pi\in\Pi_{d,k}$, almost surely 
\begin{equation}
\dim_H \pi \Phi\bar{\mu} +\dim_H \Phi\bar{\mu} _{y,\pi}=\dim_H \Phi\bar{\mu} \quad \mbox{ for } \pi \Phi\bar{\mu}\mbox{ a.e. } y\in\pi (K),\label{corcor}
\end{equation}
that is $\pi$ is dimension conserving. In particular, if $\bar{\mu} $ 
is deterministic (i.e. a self-similar measure), then  (\ref{corcor}) holds for all $\pi\in\Pi_{d,k}$.
\end{corollary}

\subsection{Proof of Theorem \ref{thmed}(i)}

The proof is adapted from \cite{FeHu09}. Recall that $R=\max\{|x|:x\in K\}$. For $n\ge 0$ and $\underline{i}\in\Lambda^{\mathbb{N}}$ let
\[
B_\Phi(\underline{i},n)=\Phi^{-1}\left(B(\Phi(\underline{i}),R\cdot r_{\underline{i}|_n})\right),
\]
with the convention that $r_{\emptyset}=1$. By definition we have $B_\Phi(\underline{i},0)=\Lambda^{\mathbb{N}}$ for all $\underline{i}\in\Lambda^\mathbb{N}$.

For $n\ge 1$ let
\[
f_n: \Lambda^{\mathbb{N}} \times \Omega_* \ni (\underline{i},\omega) \mapsto -\log \frac{\bar{\mu}\left(B_\Phi(\underline{i},n)\cap \mathcal{P}(\underline{i})\right)}{\bar{\mu}\left(B_\Phi(\underline{i},n)\right)} \in\mathbb{R}.
\]
Applying Proposition \ref{fenghu} in the case of $\eta=\mathcal{N}$ and $\varphi=\Phi$ we have that given any $\omega\in\Omega^*$ such that $\|\mu\|>0$, for $\mu$-a.e. $\underline{i}\in\Lambda^{\mathbb{N}}$,
\[
\lim_{n\to \infty} f_n(\underline{i},\omega) =\mathbf{I}_{\bar{\mu}}(\mathcal{P} \,|\,\mathcal{B}_\Phi)(\underline{i}):=f(\underline{i},\omega).
\]
Furthermore, as $\mathcal{P}$ is a finite partition of $m$ elements and $\bar{\mu}$ is a probability measure, setting
\[
\bar{f}(\underline{i},\omega)=-\inf_{n\ge 1}\log \frac{\bar{\mu}\left(B_\Phi(\underline{i},n)\cap \mathcal{P}(\underline{i})\right)}{\bar{\mu}\left(B_\Phi(\underline{i},n)\right)},
\]
we have
\[
\int_{\Lambda^{\mathbb{N}}} \bar{f}(\underline{i},\omega)\, \bar{\mu}(\mathrm{d}\underline{i},\omega) \le \mathbf{H}_{\bar{\mu}}(\mathcal{P})+C_d\le \log m+C_d.
\]
This implies that $\bar{f}\in L^1(\mathbb{Q})$.

\medskip

Next we apply the following ergodic theorem due to Maker \cite{Maker40}.

\begin{theorem}\label{etM}
Let $(X,\mathcal{B},\mu,T)$ be a measure-preserving system and let  $\{f_n\}$ be integrable functions on $(X,\mathcal{B},\mu)$. If $f_n(x)\to f(x)$ a.e. and if $\sup_n |f_n(x)|=\bar{f}(x)$ is integrable, then for a.e. $x$,
\[
\lim \frac{1}{n} \sum_{k=0}^{n-1} f_{n-k}\circ T^k(x) =f_\infty(x),
\]
where $f_\infty(x)=\lim \frac{1}{n}\sum_{k=0}^{n-1} f\circ T^k(x)$.
\end{theorem}

\begin{lemma}\label{fnf}
$\mathbb{P}_*$-a.s. for $\mu$-a.e. $\underline{i}\in\Lambda^{\mathbb{N}}$,
\[
\lim_{n\to\infty}-\frac{1}{n}\log \prod_{k=0}^{n-1}\frac{\bar{\mu}^{[\underline{i}|_{k}]}\left(B_\Phi(\sigma^k\underline{i},n-k)\cap \mathcal{P}(\sigma^{k}\underline{i})\right)}{\bar{\mu}^{[\underline{i}|_k]}\left( B_\Phi(\sigma^k\underline{i},n-k)\right)}=\mathbb{E}(\mathbf{H}_{\bar{\mu}}(\mathcal{P} \, |\, \mathcal{B}_\Phi)).
\] 
\end{lemma}

\begin{proof}
First notice that $\displaystyle -\log \frac{\bar{\mu}^{[\underline{i}|_{k}]}\left(B_\Phi(\sigma^k\underline{i},n-k)\cap \mathcal{P}(\sigma^{k}\underline{i})\right)}{\bar{\mu}^{[\underline{i}|_k]}\left( B_\Phi(\sigma^k\underline{i},n-k)\right)}= f_{n-k}\circ T^k(\underline{i},\omega)$. From Theorem \ref{etM}, for $\mathbb{Q}$-a.e. $(\underline{i},\omega)$,
\[
\lim_{n\to\infty} \frac{1}{n}\sum_{k=0}^{n-1} f_{n-k}\circ T^k(\underline{i},\omega)=f_\infty(\underline{i},\omega),
\]
where $f_\infty(\underline{i},\omega)=\lim_{n\to\infty}\frac{1}{n}\sum_{k=0}^{n-1} f\circ T^k(\underline{i},\omega)$.  But for $\mathbb{Q}$-a.e. $(\underline{i},\omega)\in \Omega'$, $f_\infty(\underline{i},\omega)=\mathbb{E}_\mathbb{Q}(f)=\mathbb{E}(\mathbf{H}_{\bar{\mu}}(\mathcal{P} \, |\, \mathcal{B}_\Phi))$  by Proposition \ref{ge}, hence the conclusion.
\end{proof}

The next lemma, an analogue of  \cite[Lemma 5.3]{FeHu09} for self-similar sets, relates the shift on symbolic space to its geometric effect on balls in $\mathbb{R}^d$.

\begin{lemma}\label{bphi}
For $\underline{i}\in\Lambda^{\mathbb{N}}$ and $r>0$ we have
\[
\Phi^{-1}\left(B(\Phi(\underline{i}),r_{\underline{i}|_1}\! \cdot r)\right)\cap \mathcal{P}(\underline{i})=\sigma^{-1} \Phi^{-1}\left(B(\Phi(\sigma\underline{i}),r)\right)\cap \mathcal{P}(\underline{i}).
\]
\end{lemma}
\begin{proof}
For $\underline{i}=i_1i_2\cdots $ and $r>0$ we have
\[
B(\Phi(\underline{i}),r_{\underline{i}|_1} \!\cdot r)=f_{i_1} \left(B(\Phi(\sigma\underline{i}), r)\right).
\]
Thus
\[
\Phi^{-1}\left(B(\Phi(\underline{i}),r_{\underline{i}|_1}\! \cdot r)\right)\cap \mathcal{P}(\underline{i}) = \Phi^{-1}\left(f_{i_1} \left(B(\Phi(\sigma\underline{i}), r)\right)\right)\cap \mathcal{P}(\underline{i}).
\]
As
\begin{eqnarray*}
 \underline{j}=j_1j_2\cdots \in \Phi^{-1}\left(f_{i_1} \left(B(\Phi(\sigma\underline{i}), r)\right)\right)\cap \mathcal{P}(\underline{i})
&\Leftrightarrow& j_1=i_1, \ \Phi(\underline{j}) \in f_{i_1} \left(B(\Phi(\sigma\underline{i}), r)\right)\\
&\Leftrightarrow&  j_1=i_1, \ f_{j_1}\left(\Phi(\sigma\underline{j})\right) \in f_{i_1} \left(B(\Phi(\sigma\underline{i}), r)\right)\\
&\Leftrightarrow&  j_1=i_1, \ \Phi(\sigma\underline{j}) \in B(\Phi(\sigma\underline{i}), r)\\
&\Leftrightarrow&  j_1=i_1, \ \underline{j} \in \sigma^{-1}\Phi^{-1}\left(B(\Phi(\sigma\underline{i}), r)\right)\\
&\Leftrightarrow& \underline{j} \in \sigma^{-1}\Phi^{-1}\left(B(\Phi(\sigma\underline{i}), r)\right)\cap \mathcal{P}(\underline{i})
\end{eqnarray*}
we get $ \Phi^{-1}\left(f_{i_1} \left(B(\Phi(\sigma\underline{i}), r)\right)\right)\cap \mathcal{P}(\underline{i})=\sigma^{-1} \Phi^{-1}\left(B(\Phi(\sigma\underline{i}),r)\right)\cap \mathcal{P}(\underline{i})$, hence the conclusion.
\end{proof}

For $\underline{i}\in\Lambda^{\mathbb{N}}$ and $n\ge 1$, conditioning on $\mu([\underline{i}|_n])>0$, we obtain
\begin{align}
&\frac{\mu\left(B_\Phi(\underline{i},n)\right)}{ \mu^{[\underline{i}|_n]}\left( B_\Phi(\sigma^n\underline{i},0)\right)}\\
&= \prod_{k=0}^{n-1}\frac{\mu^{[\underline{i}|_k]}\left( B_\Phi(\sigma^k\underline{i},n-k)\right)}{\mu^{[\underline{i}|_{k+1}]}\left( B_\Phi(\sigma^{k+1}\underline{i},n-k-1)\right)} \nonumber\\
&= \prod_{k=0}^{n-1}\frac{\mu^{[\underline{i}|_k]}\left( B_\Phi(\sigma^k\underline{i},n-k)\right)}{\mu^{[\underline{i}|_{k}]}\left(B_\Phi(\sigma^k\underline{i},n-k)\cap \mathcal{P}(\sigma^{k}\underline{i})\right)}  \frac{\mu^{[\underline{i}|_{k}]}\left(B_\Phi(\sigma^k\underline{i},n-k)\cap \mathcal{P}(\sigma^{k}\underline{i})\right)}{\mu^{[\underline{i}|_{k+1}]}\left(B_\Phi(\sigma^{k+1}\underline{i},n-k-1)\right)}\nonumber\\
&= \prod_{k=0}^{n-1}\frac{\mu^{[\underline{i}|_k]}\left( B_\Phi(\sigma^k\underline{i},n-k)\right)}{\mu^{[\underline{i}|_{k}]}\left(B_\Phi(\sigma^k\underline{i},n-k)\cap \mathcal{P}(\sigma^{k}\underline{i})\right)}  \frac{\mu^{[\underline{i}|_{k}]}\left(\sigma^{-1}B_\Phi(\sigma^{k+1}\underline{i},n-k-1)\cap \mathcal{P}(\sigma^{k}\underline{i})\right)}{\mu^{[\underline{i}|_{k+1}]}\left(B_\Phi(\sigma^{k+1}\underline{i},n-k-1)\right)}\nonumber\\
&= \prod_{k=0}^{n-1}\frac{\mu^{[\underline{i}|_k]}\left( B_\Phi(\sigma^k\underline{i},n-k)\right)}{\mu^{[\underline{i}|_{k}]}\left(B_\Phi(\sigma^k\underline{i},n-k)\cap \mathcal{P}(\sigma^{k}\underline{i})\right)}\cdot  W^{[\underline{i}|_k]}_{i_{k+1}}\nonumber\\
&= \prod_{k=0}^{n-1}\frac{\bar{\mu}^{[\underline{i}|_k]}\left( B_\Phi(\sigma^k\underline{i},n-k)\right)}{\bar{\mu}^{[\underline{i}|_{k}]}\left(B_\Phi(\sigma^k\underline{i},n-k)\cap \mathcal{P}(\sigma^{k}\underline{i})\right)}\cdot  W^{[\underline{i}|_k]}_{i_{k+1}}.\label{eqm1}
\end{align}

To complete the proof of (i) we need the following lemma.

\begin{lemma}\label{123}
$\mathbb{P}_*$-a.s. for $\mu$-a.e. $\underline{i}\in\Lambda^{\mathbb{N}}$,
\begin{itemize}
\item[(1)] $\lim_{n\to\infty}\frac{1}{n}\log  \prod_{k=0}^{n-1} W^{[\underline{i}|_k]}_{i_{k+1}}=\sum_{i=1}^m\mathbb{E}(W_i\log W_i)$;
\item[(2)] $\lim_{n\to\infty}\frac{1}{n}\log  r_{\underline{i}|_n}=\sum_{i=1}^m\mathbb{E}(W_i)\log r_i$;
\item[(3)] $\lim_{n\to\infty}\frac{1}{n}\log \|\mu^{[\underline{i}|_n]}\|=0$.
\end{itemize}
\end{lemma}

\begin{proof}
(1) and (2) follow from the strong law of large numbers under the Peyri\`ere measure $\mathbb{Q}$. (3) follows from \cite[Theorem IV(ii)]{Ba99}.
\end{proof}

Combining Lemma \ref{fnf}, \eqref{eqm1} and Lemma \ref{123} we have proved that $\mathbb{P}_*$-a.s. for $\mu$-a.e. $\underline{i}\in\Lambda^{\mathbb{N}}$,
\[
\lim_{n\to\infty}\frac{\log \Phi\mu(B(\Phi(\underline{i}),r_{\underline{i}|_n}))}{\log r_{\underline{i}|_n}}=\frac{\mathbb{E}(\mathbf{H}_{\bar{\mu}}(\mathcal{P} \, |\, \mathcal{B}_\Phi))+\sum_{i=1}^m\mathbb{E}(W_i\log W_i)}{\sum_{i=1}^m\mathbb{E}(W_i)\log r_i},
\]
which gives the conclusion.

\subsection{Proof of Theorem \ref{thmed}(ii)}

The proof is analogous to that of Theorem \ref{thmed}(i); we can formally replace $\Phi$ by $\pi g\Phi$. Here we only present the differences.  For $n\ge 0$, $g\in G$ and $\underline{i}\in\Lambda^{\mathbb{N}}$ let
\[
B_{\pi g\Phi}(\underline{i},n)=(\pi g\Phi)^{-1}\left(B(\pi g \Phi(\underline{i}),R\cdot r_{\underline{i}|_n})\right).
\]
Notice that $B_{\pi g\Phi}(\underline{i},0)=\Lambda^\mathbb{N}$ for all $\underline{i}\in\Lambda^{\mathbb{N}}$. For $n\ge 1$ let
\[
f_n: \Lambda^{\mathbb{N}} \times \Omega_* \times G \ni (\underline{i},\omega,g) \mapsto -\log \frac{\bar{\mu}\left(B_{\pi g\Phi}(\underline{i},n)\cap \mathcal{P}(\underline{i})\right)}{\bar{\mu}\left(B_{\pi g\Phi}(\underline{i},n)\right)} \in\mathbb{R}.
\]
Using Proposition \ref{fenghu} again in the case of $\eta=\mathcal{N}$ and $\varphi=\pi g \Phi$ we get that given any $\omega\in\Omega^*$ such that $\|\mu\|>0$ and given any $g\in G$, for $\mu$-a.e. $\underline{i}\in\Lambda^{\mathbb{N}}$,
\begin{equation}\label{fnpg}
\lim_{n\to \infty} f_n(\underline{i},\omega,g) =\mathbf{I}_{\bar{\mu}}(\mathcal{P} \,|\,\mathcal{B}_{\pi g\Phi})(\underline{i}):=f(\underline{i},\omega,g).
\end{equation}
Furthermore,
\[
\int_{\Lambda^{\mathbb{N}}} \sup_n |f_n(\underline{i},\omega, g)|\, \bar{\mu}(\mathrm{d}\underline{i},\omega) \le \mathbf{H}_{\bar{\mu}}(\mathcal{P})+C_d\le \log m+C_d.
\]
This implies that $\sup_n |f_n|\in L^1(\mathbb{Q}\times \xi)$. By using Theorem \ref{etM} and Proposition \ref{gee} it follows that $\mathbb{P}_*$-a.s. for $\xi$-a.e. $g\in G$ and $\mu$-a.e. $\underline{i}\in\Lambda^{\mathbb{N}}$,
\begin{equation}\label{fnf2}
\lim_{n\to\infty} \frac{1}{n}\sum_{k=0}^{n-1}f_{n-k}\circ T_\phi^k(\underline{i},\omega,g)=\mathbb{E}_{\mathbb{Q}\times \xi}(f)=\mathbb{E}_{\mathbb{P}^*\times \xi}(\mathbf{H}_{\bar{\mu}}(\mathcal{P} \, |\, \mathcal{B}_{\pi g\Phi})).
\end{equation}

The following is an analogue of Lemma \ref{bphi}.

\begin{lemma}\label{bpiphi}
For $\underline{i}\in\Lambda^{\mathbb{N}}$, $g\in G$ and $r>0$ we have
\[
(\pi g\Phi)^{-1}\left(B(\pi g\Phi(\underline{i}),r_{\underline{i}|_1}\!\cdot r)\right)\cap \mathcal{P}(\underline{i})=\sigma^{-1} (\pi g O_{\underline{i}|_1}\Phi)^{-1}\left(B(\pi g O_{\underline{i}|_1}\Phi(\sigma\underline{i}),r)\right)\cap \mathcal{P}(\underline{i}).
\]
\end{lemma}
\begin{proof}
For $\underline{i}=i_1i_2\cdots $ and $r>0$ we have
\[
B(\pi g \Phi(\underline{i}),r_{\underline{i}|_1}\!\cdot r)=\pi  g f_{i_1} \left(B(\Phi(\sigma\underline{i}), r)\right).
\]
Thus
\[
(\pi g \Phi)^{-1}\left(B(\pi g\Phi(\underline{i}),r_{\underline{i}|_1}\!\cdot r)\right)\cap \mathcal{P}(\underline{i}) = (\pi g \Phi)^{-1}\left(\pi g f_{i_1} \left(B(\Phi(\sigma\underline{i}), r)\right)\right)\cap \mathcal{P}(\underline{i}).
\]
But
\begin{eqnarray*}
&& \underline{j}=j_1j_2\cdots \in(\pi g \Phi)^{-1}\left(\pi g f_{i_1} \left(B(\Phi(\sigma\underline{i}), r)\right)\right)\cap \mathcal{P}(\underline{i})\\
&\Leftrightarrow& j_1=i_1, \ \Phi(\underline{j}) \in (\pi g)^{-1}\left(\pi g f_{i_1} \left(B(\Phi(\sigma\underline{i}), r)\right)\right)\\
&\Leftrightarrow&  j_1=i_1, \ f_{j_1}\left(\Phi(\sigma\underline{j})\right) \in (\pi g)^{-1}\left(\pi g f_{i_1} \left(B(\Phi(\sigma\underline{i}), r)\right)\right)\\
&\Leftrightarrow&  j_1=i_1, \ \pi g f_{j_1}\left(\Phi(\sigma\underline{j})\right) \in \pi g f_{i_1} \left(B(\Phi(\sigma\underline{i}), r)\right)\\
&\Leftrightarrow&  j_1=i_1, \ \pi g O_{j_1}\left(\Phi(\sigma\underline{j})\right) \in \pi g O_{i_1} \left(B(\Phi(\sigma\underline{i}), r)\right)\\
&\Leftrightarrow&  j_1=i_1, \ \pi g O_{j_1} \Phi(\sigma\underline{j}) \in B(\pi g O_{i_1}\Phi(\sigma\underline{i}), r)\\
&\Leftrightarrow&  j_1=i_1, \ \underline{j} \in \sigma^{-1}(\pi g O_{j_1} \Phi)^{-1}\left(B(\pi g O_{i_1} \Phi(\sigma\underline{i}), r)\right)\\
&\Leftrightarrow& \underline{j} \in \sigma^{-1}(\pi g O_{\underline{i}|_1}\Phi)^{-1}\left(B(\pi g O_{\underline{i}|_1}\Phi(\sigma\underline{i}), r)\right)\cap \mathcal{P}(\underline{i}),
\end{eqnarray*}
which gives the conclusion.
\end{proof}

For $\underline{i}\in\Lambda^{\mathbb{N}}$ and $n\ge 1$, conditioning on $\mu([\underline{i}|_n])>0$,\begin{align}
&\frac{\mu\big(B_{\pi g\Phi}(\underline{i},n)\big)}{\mu^{[\underline{i}|_n]}\big( B_{\pi g O_{\underline{i}|_n}\Phi}(\sigma^n\underline{i},0)\big)}\nonumber\\
&= \prod_{k=0}^{n-1}\frac{\mu^{[\underline{i}|_k]}\big( B_{\pi g O_{\underline{i}|_k}\Phi}(\sigma^k\underline{i},n-k)\big)}{\mu^{[\underline{i}|_{k+1}]}\big( B_{\pi g O_{\underline{i}|_{k+1}}\Phi}(\sigma^{k+1}\underline{i},n-k-1)\big)}\nonumber\\
&= \prod_{k=0}^{n-1}\frac{\mu^{[\underline{i}|_k]}\big( B_{\pi g O_{\underline{i}|_k}\Phi}(\sigma^k\underline{i},n-k)\big)}{\mu^{[\underline{i}|_{k}]}\big(B_{\pi g O_{\underline{i}|_{k}}\Phi}(\sigma^k\underline{i},n-k)\cap \mathcal{P}(\sigma^{k}\underline{i})\big)} \frac{\mu^{[\underline{i}|_{k}]}\big(B_{\pi g O_{\underline{i}|_k}\Phi}(\sigma^k\underline{i},n-k)\cap \mathcal{P}(\sigma^{k}\underline{i})\big)}{\mu^{[\underline{i}|_{k+1}]}\big(B_{\pi g O_{\underline{i}|_{k+1}}\Phi}(\sigma^{k+1}\underline{i},n-k-1)\big)}\nonumber\\
&= \prod_{k=0}^{n-1}\frac{\mu^{[\underline{i}|_k]}\big( B_{\pi g O_{\underline{i}|_k}\Phi}(\sigma^k\underline{i},n-k)\big)}{\mu^{[\underline{i}|_{k}]}\big(B_{\pi g O_{\underline{i}|_k}\Phi}(\sigma^k\underline{i},n-k)\cap \mathcal{P}(\sigma^{k}\underline{i})\big)}\frac{\mu^{[\underline{i}|_{k}]}\big(\sigma^{-1}B_{\pi g O_{\underline{i}|_{k+1}}\Phi}(\sigma^{k+1}\underline{i},n-k-1)\cap \mathcal{P}(\sigma^{k}\underline{i})\big)}{\mu^{[\underline{i}|_{k+1}]}\big(B_{\pi g O_{\underline{i}|_{k+1}}\Phi}(\sigma^{k+1}\underline{i},n-k-1)\big)}\nonumber\\
&= \prod_{k=0}^{n-1}\frac{\mu^{[\underline{i}|_k]}\big( B_{\pi g O_{\underline{i}|_k}\Phi}(\sigma^k\underline{i},n-k)\big)}{\mu^{[\underline{i}|_{k}]}\big(B_{\pi g O_{\underline{i}|_k}\Phi}(\sigma^k\underline{i},n-k)\cap \mathcal{P}(\sigma^{k}\underline{i})\big)}\cdot  W^{[\underline{i}|_k]}_{i_{k+1}}\nonumber\\
&= \prod_{k=0}^{n-1}\frac{\bar{\mu}^{[\underline{i}|_k]}\big( B_{\pi g O_{\underline{i}|_k}\Phi}(\sigma^k\underline{i},n-k)\big)}{\bar{\mu}^{[\underline{i}|_{k}]}\big(B_{\pi g O_{\underline{i}|_k}\Phi}(\sigma^k\underline{i},n-k)\cap \mathcal{P}(\sigma^{k}\underline{i})\big)}\cdot  W^{[\underline{i}|_k]}_{i_{k+1}}.\label{eqm1'}
\end{align}
Notice that for $k\ge 0$,
\[
f_{n-k}\circ T_\phi^k(\underline{i},\omega,g)=\log\frac{\bar{\mu}^{[\underline{i}|_k]}\big( B_{\pi g O_{\underline{i}|_k}\Phi}(\sigma^k\underline{i},n-k)\big)}{\bar{\mu}^{[\underline{i}|_{k}]}\big(B_{\pi g O_{\underline{i}|_k}\Phi}(\sigma^k\underline{i},n-k)\cap \mathcal{P}(\sigma^{k}\underline{i})\big)}.
\]
Using \eqref{fnf2} we conclude that $\mathbb{P}_*$-a.s. for $\xi$-a.e. $g\in G$ and $\mu$-a.e. $\underline{i}\in\Lambda^{\mathbb{N}}$,
\begin{equation}\label{eqm2}
\lim_{n\to\infty}\frac{1}{n}\log \prod_{k=0}^{n-1}\frac{\bar{\mu}^{[\underline{i}|_k]}\big( B_{\pi g O_{\underline{i}|_k}\Phi}(\sigma^k\underline{i},n-k)\big)}{\bar{\mu}^{[\underline{i}|_{k}]}\big(B_{\pi g O_{\underline{i}|_k}\Phi}(\sigma^k\underline{i},n-k)\cap \mathcal{P}(\sigma^{k}\underline{i})\big)}
=\mathbb{E}_{\mathbb{P}^*\times \xi}(\mathbf{H}_{\bar{\mu}}(\mathcal{P} \, |\, \mathcal{B}_{\pi g\Phi})).
\end{equation}
This completes the proof.

\subsection{Proof of Theorem \ref{thmed}(iii)}

\begin{proof}
Given $k\ge 1$, $g\in G$ and $\bar{\mu}>0$, Lemma \ref{condimea} yields that for $\mu$-a.e. $\underline{i}\in\Lambda^{\mathbb{N}}$
\[
\bar{\mu}_{\underline{i},\pi g}(B_{\Phi}(\underline{i},k)\cap \mathcal{P}(\underline{i}))=\lim_{n\to\infty} \frac{\bar{\mu}(B_{\Phi}(\underline{i},k)\cap B_{\pi g\Phi}(\underline{i},n)\cap \mathcal{P}(\underline{i}))}{\bar{\mu}(B_{\pi g\Phi}(\underline{i},n))}.
\]
From Lemmas \ref{bphi} and \ref{bpiphi} we get 
\begin{align*}
&\frac{\bar{\mu}(B_{\Phi}(\underline{i},k)\cap B_{\pi g\Phi}(\underline{i},n)\cap \mathcal{P}(\underline{i}))}{\bar{\mu}(B_{\pi g\Phi}(\underline{i},n))}\\
=&\frac{\bar{\mu}(B_{\Phi}(\underline{i},k)\cap B_{\pi g\Phi}(\underline{i},n)\cap \mathcal{P}(\underline{i}))}{\bar{\mu}( B_{\pi g\Phi}(\underline{i},n)\cap \mathcal{P}(\underline{i}))}\frac{\bar{\mu}(B_{\pi g\Phi}(\underline{i},n)\cap \mathcal{P}(\underline{i}))}{\bar{\mu}(B_{\pi g\Phi}(\underline{i},n))}\\
=&\frac{\bar{\mu}^{[\underline{i}|_1]}(B_{\Phi}(\sigma\underline{i},k-1)\cap B_{\pi gO_{\underline{i}|_1}\Phi}(\sigma\underline{i},n-1))}{\bar{\mu}^{[\underline{i}|_1]}( B_{\pi gO_{\underline{i}|_1}\Phi}(\sigma\underline{i},n-1))}\frac{\bar{\mu}(B_{\pi g\Phi}(\underline{i},n)\cap \mathcal{P}(\underline{i}))}{\bar{\mu}(B_{\pi g\Phi}(\underline{i},n))}.
\end{align*}
Since $\bar{\mu}^{[\underline{i}|_1]}$ is absolutely continuous with respect to $\sigma\bar{\mu}|_{[\underline{i}|_1]}$, we obtain, in a similar way to the proof of Lemma \ref{condimea}, that for $\mu$-a.e. $\underline{i}\in\Lambda^{\mathbb{N}}$,
\[
\lim_{n\to\infty} \frac{\bar{\mu}^{[\underline{i}|_1]}(B_{\Phi}(\sigma\underline{i},k-1)\cap B_{\pi gO_{\underline{i}|_1}\Phi}(\sigma\underline{i},n-1))}{\bar{\mu}^{[\underline{i}|_1]}( B_{\pi gO_{\underline{i}|_1}\Phi}(\sigma\underline{i},n-1))}=\bar{\mu}_{\sigma\underline{i},\pi g O_{\underline{i}|_1}}^{[\underline{i}|_1]}(B_{\Phi}(\sigma\underline{i},k-1)).
\]
On the other hand, by \eqref{fnpg},
\[
\lim_{n\to\infty}\frac{\bar{\mu}(B_{\pi g\Phi}(\underline{i},n)\cap \mathcal{P}(\underline{i}))}{\bar{\mu}(B_{\pi g\Phi}(\underline{i},n))}=\exp(-\mathbf{I}_{\bar{\mu}}(\mathcal{P} \,|\,\mathcal{B}_{\pi g\Phi})(\underline{i})).
\] 
Hence, for $k\ge 1$, $\mathbb{P}_*$ a.s. for $\xi$-a.e. $g\in G$ and $\mu$-a.e. $\underline{i}\in\Lambda^{\mathbb{N}}$,
\begin{align*}
\bar{\mu}_{\underline{i},\pi g}(B_{\Phi}(\underline{i},k)\cap \mathcal{P}(\underline{i}))=\bar{\mu}_{\sigma\underline{i},\pi g O_{\underline{i}|_1}}^{[\underline{i}|_1]}(B_{\Phi}(\sigma\underline{i},k-1))\cdot \exp(-\mathbf{I}_{\bar{\mu}}(\mathcal{P} \,|\,\mathcal{B}_{\pi g\Phi})(\underline{i})).
\end{align*}
This gives, noting that $\bar{\mu}_{\sigma^n\underline{i},\pi gO_{\underline{i}|_n}}^{[\underline{i}|_n]}(B_{\Phi}(\sigma^n\underline{i},0))=1$,
\begin{equation}\label{cm}
\bar{\mu}_{\underline{i},\pi g}(B_{\Phi}(\underline{i},n))
=\prod_{k=0}^{n-1} \frac{\bar{\mu}_{\sigma^k\underline{i},\pi gO_{\underline{i}|_k}}^{[\underline{i}|_k]}(B_{\Phi}(\sigma^k\underline{i},n-k))}{\bar{\mu}_{\sigma^k\underline{i},\pi gO_{\underline{i}|_k}}^{[\underline{i}|_k]}(B_{\Phi}(\underline{i},n-k)\cap \mathcal{P}(\sigma^k\underline{i}))}  \exp({-\mathbf{I}_{\bar{\mu}^{[\underline{i}|_k]}}(\mathcal{P} \,|\,\mathcal{B}_{\pi g O_{\underline{i}|_k}\Phi})(\sigma^k\underline{i})}).
\end{equation}

We need the following lemma:
\begin{lemma}\label{fnf'}
$\mathbb{P}_*$ a.s. for $\xi$-a.e. $g\in G$ and $\mu$-a.e. $\underline{i}\in\Lambda^{\mathbb{N}}$,
\[
\lim_{n\to \infty}\frac{1}{n}\log\prod_{k=0}^{n-1} \frac{\bar{\mu}_{\sigma^k\underline{i},\pi gO_{\underline{i}|_k}}^{[\underline{i}|_k]}(B_{\Phi}(\sigma^k\underline{i},n-k))}{\bar{\mu}_{\sigma^k\underline{i},\pi gO_{\underline{i}|_k}}^{[\underline{i}|_k]}(B_{\Phi}(\underline{i},n-k)\cap \mathcal{P}(\sigma^k\underline{i}))} =\mathbb{E}(\mathbf{H}_{\bar{\mu}}(\mathcal{P} \, |\,  \mathcal{B}_{\Phi})).
\]
\end{lemma}

\begin{proof}
For $n\ge 1$ let
\[
f_n(\underline{i},\omega,g)=-\log \frac{\bar{\mu}_{\underline{i},\pi g}(B_{\Phi}(\underline{i},n)\cap \mathcal{P}(\underline{i}))}{\bar{\mu}_{\underline{i},\pi g}(B_{\Phi}(\underline{i},n))}.
\]
Applying Proposition \ref{fenghu} in the case of $\eta=\mathcal{P}_{\pi g}$ and $\varphi=\Phi$ we get that given $g\in G$, for $\mathbb{Q}$-a.e. $(\underline{i},\omega)\in\Omega'$ the sequence $f_n$ converges to
\[
f:=\mathbf{I}_{\bar{\mu}}(\mathcal{P}\,|\, \widehat{\mathcal{P}_{\pi g}}\vee \mathcal{B}_{\Phi})=\mathbf{I}_{\bar{\mu}}(\mathcal{P}\,|\,  \mathcal{B}_\Phi),
\]
Here we have used that the $\sigma$-algebra $\widehat{\mathcal{P}_{\pi g}}$ is a sub-$\sigma$-algebra of $\mathcal{B}_{\Phi}$. Moreover,  since $\int \sup_{n} |f_n| \, \mathrm{d}\bar{\mu}\le \mathbf{H}_{\bar{\mu}}(\mathcal{P})+C_d\le \log m+C_d$, $\sup_{n} |f_n|$ is integrable. As
\[
\frac{1}{n}\log\prod_{k=0}^{n-1} \frac{\bar{\mu}_{\sigma^k\underline{i},\pi gO_{\underline{i}|_k}}^{[\underline{i}|_k]}(B_{\Phi}(\sigma^k\underline{i},n-k))}{\bar{\mu}_{\sigma^k\underline{i},\pi gO_{\underline{i}|_k}}^{[\underline{i}|_k]}(B_{\Phi}(\underline{i},n-k)\cap \mathcal{P}(\sigma^k\underline{i}))}=\frac{1}{n}\sum_{k=0}^{n-1} f_{n-k}\circ T_\phi^k(\underline{i},\omega,g),
\]
the conclusion follows from Theorem \ref{etM} and Proposition \ref{gee}.
\end{proof}

By Proposition \ref{gee} we have $\mathbb{P}_*$ a.s. for $\xi$-a.e. $g\in G$ and $\mu$-a.e. $\underline{i}\in\Lambda^{\mathbb{N}}$,
\[
\lim_{n\to\infty}\frac{1}{n}\sum_{k=0}^{n-1} \mathbf{I}_{\bar{\mu}^{[\underline{i}|_k]}}(\mathcal{P} \,|\,\mathcal{B}_{\pi g O_{\underline{i}|_k}\Phi})(\sigma^k\underline{i})=\mathbb{E}_{\mathbb{P}^*\times \xi}(\mathbf{H}_{\bar{\mu}}(\mathcal{P} \, |\,\mathcal{B}_{\pi g\Phi}))
\]
Combining \eqref{cm} and Lemma \ref{fnf'} we get that $\mathbb{P}_*$-a.s. for $\xi$-a.e. $g\in G$ and $\mu$-a.e. $\underline{i}\in\Lambda^{\mathbb{N}}$,
\[
\lim_{n\to \infty}\frac{1}{n}\log \bar{\mu}_{\underline{i},\pi g}(B_{\Phi}(\underline{i},n))=\mathbb{E}(\mathbf{H}_{\bar{\mu}}(\mathcal{P} \, |\, \mathcal{B}_{\Phi}))-\mathbb{E}_{\mathbb{P}^*\times \xi}(\mathbf{H}_{\bar{\mu}}(\mathcal{P} \, |\,\mathcal{B}_{\pi g\Phi})),
\]
so that $\mathbb{P}_*$-a.s. for $\xi$-a.e. $g\in G$ and $\mu$-a.e. $\underline{i}\in\Lambda^{\mathbb{N}}$,
\begin{equation}\label{edcm}
\lim_{r\to 0} \frac{\log \Phi\bar{\mu}_{\underline{i},\pi g}(B(\Phi(\underline{i}),r))}{\log r}=\frac{\mathbb{E}(\mathbf{H}_{\bar{\mu}}(\mathcal{P} \, |\, \mathcal{B}_{\Phi}))-\mathbb{E}_{\mathbb{P}^*\times \xi}(\mathbf{H}_{\bar{\mu}}(\mathcal{P} \, |\,\mathcal{B}_{\pi g\Phi}))}{\sum_{i=1}^m\mathbb{E}(W_i)\log r_i}.
\end{equation}
Together with \eqref{disinter} this yields (iii).
\end{proof}

\section{Dimension of projections}\label{dip}
\setcounter{theorem}{0}

In this section we generalize the results of  \cite{HoSh12} on projections and images under $C^1$ functions without singularites to random cascade measures.

Let $D=B(0,R)$ where $R=\max\{|x|:x\in K\}$. Denote by $\mathcal{M}$ the family of probability measures on $D$ and let $\mathcal{B}_\star$ be its weak-$\star$ topology. Denote by $C(\mathcal{M})$ the family of all continuous functions on $\mathcal{M}$. We use the separability of $C(\mathcal{M})$ in $\|\cdot\|_\infty$ to get convergence of ergodic averages for all $h\in C(\mathcal{M})$.

\begin{proposition}\label{unier}
$\mathbb{P}_*$-a.s. for $\xi$-a.e. $g$ and $\mu$-a.e. $\underline{i}$,
\[
\lim_{N\to\infty} \frac{1}{N} \sum_{n=0}^{N-1}h(gO_{\underline{i}|_{n}}\Phi\bar{\mu}^{[\underline{i}|_{n}]})=\mathbb{E}_\mathbb{Q\times \xi}(h(g\Phi\bar{\mu}))
\]
 for all $h\in C(\mathcal{M})$.
\end{proposition}

\begin{proof}
Let $\{h_k\}_{k\ge 1}$ be a countable dense sequence in $C(\mathcal{M})$. If we write
\[
M:X\ni (\underline{i},\omega,g) \mapsto g\Phi \bar{\mu}\in \mathcal{M},
\]
then it is easy to verify that for $n\ge 1$
\[
M\circ T_\phi^n(\underline{i},\omega,g)=gO_{\underline{i}|_n}\Phi \bar{\mu}^{[\underline{i}|_n]}.
\]
It follows  from Proposition \ref{gee}  that $\mathbb{P}_*$-a.s. for $\xi$-a.e. $g$ and $\mu$-a.e. $\underline{i}$,
\[
\lim_{N\to\infty} \frac{1}{N} \sum_{n=0}^{N-1}h_k(gO_{\underline{i}|_{n}}\Phi\bar{\mu}^{[\underline{i}|_{n}]})=\mathbb{E}_\mathbb{Q\times\xi}(h_k(g\Phi\bar{\mu})) \quad\mbox{ for all } k\ge 1.
\]
For any $h\in C(\mathcal{M})$, take a subsequence $\{h_k'\}_{k\ge 1}$ of $\{h_k\}_{k\ge 1}$ that converges to $h$. On the one hand, since $\mathcal{M}$ is compact, $h$ is bounded, so by the uniform convergence in $\|\cdot \|_\infty$,
\[
\lim_{k\to\infty} \mathbb{E}_\mathbb{Q\times\xi}(h_k'(g\Phi\bar{\mu}))=\mathbb{E}_\mathbb{Q\times\xi}(h(g\Phi\bar{\mu})).
\]
On the other hand, for each $N$,
\[
\left|\frac{1}{N} \sum_{n=0}^{N-1}h_k'(gO_{\underline{i}|_{n}}\Phi\bar{\mu}^{[\underline{i}|_{n}]})-\frac{1}{N} \sum_{n=0}^{N-1}h(gO_{\underline{i}|_{n}}\Phi\bar{\mu}^{[\underline{i}|_{n}]})\right|\le \|h_k'-h\|_\infty.
\]
Thus the limit
\[
\lim_{N\to\infty} \frac{1}{N} \sum_{n=0}^{N-1}h(gO_{\underline{i}|_{n}}\Phi\bar{\mu}^{[\underline{i}|_{n}]})
\]
exists and equals $\lim_{k\to\infty}\mathbb{E}_\mathbb{Q\times\xi}(h_k'(g\Phi\bar{\mu}))=\mathbb{E}_\mathbb{Q\times\xi}(h(g\Phi\bar{\mu}))$, $\mathbb{P}_*$-a.s. for $\xi$-a.e. $g$ and $\mu$-a.e. $\underline{i}$.
\end{proof}

\subsection{Lower bound for the dimension of projections}\label{lbdp}

We use the $\rho$-tree method in \cite{HoSh12} to obtain close lower bounds for the dimensions of projections of measures. Let $\rho=\max\{r_i:i\in\Lambda\}$ and $c=\min\{r_i:i\in \Lambda\}$. For $i=i_1\cdots i_n\in\Lambda^*$ write $r_i^-=r_{i_1}\cdots r_{i_{n-1}}$. For each $q\ge 1$ we redefine the alphabet used for symbolic space to obtain one for which the contraction ratios do not vary too much:
\[
\Lambda_q=\{i\in\Lambda^*: r_i^->\rho^q \text{ and } r_i \le \rho^q\}.
\]
By definition $c\rho^q<r_i\le \rho^q$ for all $i\in \Lambda_q$. The canonical mapping $\Phi_q:(\Lambda_q^\mathbb{N},d_{\rho^q}) \mapsto K$ is $R$-Lipschitz where $R=\max\{|x|:x\in K\}$. Setting $\{W^{[j]}_q=(Q_i^{[j]})_{i\in \Lambda_q}: j\in\Lambda_q^*\}$ gives a random cascade measure $\mu_q$ on $\Lambda_q^\mathbb{N}$. Observe that it is the same random cascade measure as $\mu$ on embedding $\Lambda_q^\mathbb{N}$ into $\Lambda^\mathbb{N}$. (The slight ambiguity in notation should not cause any confusion: the subscript $q$ will always refer to the parameter redefining the alphabet, so, for example, $W_{q,i}$ refers to the element of $ W_q\equiv W^{[\emptyset]}_q $ with index $i \in \Lambda_q$.)

Let $G_q=\overline{\langle O_i:i\in\Lambda_q\rangle}$ and let $\xi_q$ be its normalised Haar measure. As before, $\Pi_{d,k}$ is the set of orthogonal projections from $\mathbb{R}^d$ onto its $k$-dimensional subspaces.

For $\pi \in \Pi_{d,k}$, $q\in\mathbb{N}$ and $\nu$ a measure on $\mathbb{R}^d$, define
\[
e_q(\pi,\nu)=\frac{1}{q\log(1/ \rho)}H_{\rho^q}(\pi\nu).
\]
So $e_q:\Pi_{d,k}\times \mathcal{M}\mapsto [0,k]$ is lower semicontinuous. Let $E_q(\pi)=\mathbb{E}_{\mathbb{P}^*\times \xi_q}(e_q(\pi,g\Phi\bar{\mu}))$. 

\begin{theorem}\label{eq}
$\mathbb{P}_*$-a.s. for $\xi_q$-a.e. $g\in G_q$,
\[
\dim_H(\pi g\Phi\bar{\mu})\ge \frac{q\log(1/\rho)}{q\log(1/\rho)-\log c} E_q(\pi)-O(1/q) \text{ for all } \pi\in\Pi_{d,k},
\]
where the implied constant in $O(1/q)$ only depends on $\rho$, $c$, $R$ and $k$.
\end{theorem}

\begin{proof}
Applying Proposition \ref{unier} to a sequence of continuous functions approximating $e_q$ from below and using the monotone convergence theorem, we have that $\mathbb{P}_*$-a.s. for $\xi_q$-a.e. $g$ and $\mu_q$-a.e. $\underline{i}$,
\begin{equation}\label{aveeq}
\liminf \frac{1}{N}\sum_{n=1}^{N}e_q(\pi,gO_{\underline{i}|_{n}}\Phi_q\bar{\mu}_q^{[\underline{i}|_{n}]})\geq  E_q(\pi) \text{ for all } \pi\in\Pi_{d,k}.
\end{equation}
Using the strong law of large numbers we note that  $\mathbb{P}_*$-a.s. for $\mu_q$-a.e. $\underline{i}\in \Lambda_q^{\mathbb{N}}$,
\[
\lim_{n\to \infty}\frac{\log Q_{\underline{i}|_n}}{-n} =-\sum_{i\in\Lambda_q}\mathbb{E}\left(\chi_{\{W_{q,i}> 0\}}W_{q,i}\log W_{q,i}\right)\in(0,\infty),
\]
so in particular, $\mathbb{P}_*$-a.s. for $\mu_q$-a.e. $\underline{i}\in \Lambda_q^{\mathbb{N}}$, $Q_{\underline{i}|_n}>0$ for all $n\ge 1$. Identically,
\[
\chi_{\{Q_{\underline{i}|_n}>0\}}\bar{\mu}_q^{[\underline{i}|_{n}]}=\chi_{\{Q_{\underline{i}|_n}>0\}}\chi_{\{\|\mu_q^{[\underline{i}|_{n}]}\|>0\}}\cdot \frac{\mu_q^{[\underline{i}|_{n}]}}{\|\mu_q^{[\underline{i}|_{n}]}\|}=\sigma^{n}\bar{\mu}_{q,[\underline{i}|_{n}]},
\]
where
\[
\bar{\mu}_{q,[\underline{i}|_{n}]}=\chi_{\{\mu_q([\underline{i}|_{n}])>0\}}\frac{\mu_q|_{[\underline{i}|_{n}]}}{\mu_q([\underline{i}|_{n}])},
\]
so by \eqref{rescal}
\begin{eqnarray*}
H_{\rho^q}(\pi gO_{\underline{i}|_{n}}\Phi_q\chi_{\{Q_{\underline{i}|_{n}}>0\}}\bar{\mu}_q^{[\underline{i}|_{n}]})&=&H_{\rho^q}(\pi gO_{\underline{i}|_{n}}\Phi_q\sigma^{n}\bar{\mu}_{q,[\underline{i}|_{n}]})\\
&=&H_{\rho^q\cdot r_{\underline{i}|_n}}(\pi g\Phi_q\bar{\mu}_{q,[\underline{i}|_{n}]})\\
&\le &H_{(c\rho^q)^{n+1}}(\pi g\Phi_q\bar{\mu}_{q,[\underline{i}|_{n}]}).
\end{eqnarray*}
Hence,  using (\ref{aveeq}), $\mathbb{P}_*$-a.s. for $\xi_q$-a.e. $g$ and $\mu_q$-a.e. $\underline{i}$,
\[
\frac{1}{q\log(1/\rho)}\liminf_{N\to\infty}\frac{1}{N}\sum_{n=1}^N H_{(c\rho^q)^{n+1}}(\pi g\Phi_q\bar{\mu}_{q,[\underline{i}|_{n}]})\ge E_q(\pi) \text{ for all } \pi\in\Pi_{d,k}.
\]
The mapping $f\equiv\pi g\Phi_q: ((\Lambda^q)^{\mathbb{N}},d_{\rho^q})\mapsto \mathbb{R}^k$ is $R$-Lipschitz. By  \cite[Theorem 5.4]{HoSh12} there exist a $\rho^q$-tree $(X,d_{\rho^q})$ and maps $(\Lambda^q)^{\mathbb{N}}\overset{h}{\mapsto} X\overset{f'}{\mapsto} \mathbb{R}^k$ such that $f=f'h$, where $h$ is a tree morphism and $f'$ is $C$-faithful (see \cite[Definition 5.1]{HoSh12}) for some constant $C$ depending  only on $R$ and $k$. Then, applying  \cite[Proposition 5.3]{HoSh12} to the  $c\rho^q$-tree $(X,d_{c\rho^q})$ (for which $f'$ is $c^{-1}C$-faithful), there is a constant $C'$ depending only on $c^{-1}C$ and $k$ such that for all $n\ge 1$,
\[
|H_{(c\rho^q)^{n+1}}(f\bar{\mu}_{q,[\underline{i}|_{n}]})-H_{(c\rho^q)^{n+1}}(h\bar{\mu}_{q,[\underline{i}|_{n}]})|\le C'.
\]
Consequently, $\mathbb{P}_*$-a.s. for $\xi_q$-a.e. $g$ and $\bar{\mu}_q$-a.e. $\underline{i}$,
\[
\frac{1}{q\log(1/\rho)}\liminf_{N\to\infty}\frac{1}{N-1}\sum_{n=1}^N H_{(c\rho^q)^{n+1}}(h\bar{\mu}_{q,[\underline{i}|_{n}]})\ge E_q(\pi)-O(1/q) \text{ for all } \pi\in\Pi_{d,k},
\]
where the constant in $O(1/q)$ only depends on $\rho$ and $C'$. By \cite[Theorem 4.4]{HoSh12} it follows  that $\mathbb{P}_*$-a.s. for $\xi_q$-a.e. $g$,
\[
\dim_H h\bar{\mu}_q\ge \frac{q\log(1/\rho)}{q\log(1/\rho)-\log c}E_q(\pi)-O(1/q) \text{ for all } \pi\in\Pi_{d,k}.
\]
Since $f'$ is $C$-faithful and $f'h\bar{\mu}_q=f\bar{\mu}_q=\pi g \Phi_q\bar{\mu}_q=\pi g\Phi\bar{\mu}$,  the conclusion follows from \cite[Proposition 5.2]{HoSh12}.
\end{proof}

\subsection{Projection theorems}

The projection results in  \cite{HoSh12} require the strong separation condition on the underlying IFS ${\mathcal I}$. With the approach of Section \ref{lbdp} we avoid the need for any separation condition at all. Moreover, our results apply to random cascade measures as well as deterministic measures on self-similar sets.

 For the rest of this section we assume that the rotation group  $G\equiv\overline{\langle O_i:i\in \Lambda\rangle}$ is connected and we denote by $\xi$ its normalised Haar measure. We fix $\pi_0\in \Pi_{d,k}$ and write $\Pi=\pi_0 G$.

We remark that the arguments of this section extend to the more general setting where  the orbit $\pi G$ is of the form $\pi \tilde{G}$, where $\tilde{G}$ is connected. This includes the case of certain restricted families of projections, for example for projections onto the lines lying in certain cones.

\begin{lemma}\label{gq}
If $G$ is connected then
$\pi_0 G_q=\Pi$ 
for each $q\ge 1$
\end{lemma}

\begin{proof}
 For $i=i_1i_2\cdots i_l \in \Lambda^*$ (where $i_j \in \Lambda$) let ${O}_i={O}_{i_1}{O}_{i_2}\cdots {O}_{i_l}$. It is sufficient to prove that the group $H := \langle{O}_{i}:i\in \Lambda_q\rangle$ is dense in ${G}$. (Recall that the closed group generated by a set of elements coincides with the closed semigroup generated by them). 

Write $\Lambda_{<q} = \{i \in \Lambda^*:r_i > \rho^q\}.$  Then
$\bigcup_{j \in \Lambda_{<q}}  {O}_{j}H$
is dense in ${G}$. By Baire's category theorem, we may choose $j \in \Lambda_{<q}$ such that
 $\overline{{O}_{j}H}$  has nonempty interior in ${G}$. Consequently $\overline{H}$ has nonempty interior, so if $h$ is in the interior of $\overline{H}$ then $\overline{H} = h^{-1}\overline{H}$. Thus $\overline{H}$ contains a neighborhood of the identity, so since a compact connected Lie group is generated by any neighbourhood of its identity, $\overline{H}={G}$.
\end{proof}

Hence for $\pi \in \Pi$ we have
\[
E_q(\pi)=\mathbb{E}_{\mathbb{P}^*\times \xi_q}(e_q(\pi,g\Phi\bar{\mu}))=\mathbb{E}_{\mathbb{P}^*\times {\xi}}(e_q(\pi,g\Phi\bar{\mu})).
\]
For the same reason we can also deduce from Theorem \ref{thmed}(ii) that $\mathbb{P}_*$-a.s. for ${\xi}$-a.e. $g\in \tilde{G}$, $\pi_0 g\Phi\mu$ is exact-dimensional with dimension
\[
\beta(\pi_0)=\frac{\mathbb{E}_{\mathbb{P}^*\times{\xi}}(\mathbf{H}_{\bar{\mu}}(\mathcal{P} \, |\, \mathcal{B}_{\pi_0 g\Phi}))+\sum_{i=1}^m\mathbb{E}(W_i\log W_i)}{\sum_{i=1}^m\mathbb{E}(W_i)\log r_i}.
\]

\begin{theorem}\label{e}
Let $\pi_0\in \Pi_{d,k}$ and let  $G$ be connected. Then the limit
\[
E(\pi):=\lim_{q\to \infty}E_q(\pi)
\]
exists for every $\pi \in  \Pi$, and $E:  \Pi\mapsto [0,k]$ is lower semi-continuous. Moreover:
\begin{itemize}
\item[(i)] $E(\pi_0 g)=\beta(\pi_0)$ for ${\xi}$-a.e. $g$.
\item[(ii)] For a fixed $\pi\in  \Pi$, $\mathbb{P}_*$-a.s. for ${\xi}$-a.e. $g$,
\[
\dim_e\pi g\Phi\bar{\mu}=\dim_H \pi g\Phi\bar{\mu}=E(\pi).
\]
(Recall that $\dim_e$ is the entropy dimension.)
\item[(iii)] $\mathbb{P}_*$-a.s. for ${\xi}$-a.e. $g\in{G}$,
\[
\dim_H\pi g\Phi\bar{\mu} \ge E(\pi) \text{ for all } \pi\in \Pi.
\]
\end{itemize}
\end{theorem}
\begin{proof}
The proof is almost the same as that of \cite[Theorem 8.2]{HoSh12}. By Theorem \ref{eq}  and  Lemma \ref{gq} we have for each $q\ge 1$ that $\mathbb{P}_*$-a.s. for ${\xi}$-a.e. $g\in {G}$,
\[
\dim_H(\pi g\Phi\bar{\mu})\ge \frac{q\log(1/\rho)}{q\log(1/\rho)-\log c} E_q(\pi)-O(1/q) \text{ for all } \pi\in \Pi,
\]
where the implied constant in $O(1/q)$ only depends on $\rho$, $c$, $R$ and $k$. This implies that $\mathbb{P}_*$-a.s. for ${\xi}$-a.e. $g\in {G}$,
\[
\dim_H(\pi g\Phi\bar{\mu})\ge \limsup_{q\to \infty} E_q(\pi) \text{ for all } \pi\in \Pi.
\]
On the other hand by using Fatou's lemma we have
\[
\mathbb{E}_{\mathbb{P}^*\times{\xi}}(\dim_e(\pi g\Phi\bar{\mu})) \le \liminf_{q\to \infty} E_q(\pi)
\]
This implies that $\lim_{q\to\infty} E_{q}(\pi)$ exists for all $\pi \in \Pi$. Then (ii) and (iii) follow directly, and $(i)$ follows from Theorem \ref{thmed}(ii).

For the lower semicontinuity of $E$, fix $\pi\in  \Pi$ and $\epsilon>0$. Using  that $E_q(\pi) \to E(\pi)$ and $E_q$ is lower semicontinuous, there exist a number $q$ and a neighbourhood $\mathcal{U}(\pi)$ of $\pi$ in $\Pi_{d,k}$ such that for all $\pi'\in \mathcal{U}(\pi)$,
\[
 \frac{q\log(1/\rho)}{q\log(1/\rho)-\log c} E_q(\pi')-O(1/q) \ge E(\pi)-\epsilon.
\]
This gives that $\mathbb{P}_*$-a.s. for ${\xi}$-a.e. $g\in{G}$,
\[
\dim_H(\pi' g\Phi\bar{\mu})\ge E(\pi)-\epsilon \text{ for all } \pi'\in \mathcal{U}(\pi).
\]
By (ii) this yields that $E(\pi')\ge E(\pi)-\epsilon$ for all $\pi'\in \mathcal{U}(\pi)$, giving the conclusion.
\end{proof}

We can now obtain a constant lower bound for the dimension of the projected measure over all  $\pi \in \Pi$.

\begin{corollary}\label{c1}
Let  $G$ be connected and let $\pi_0 \in  \Pi_{d,k}$. Then $\mathbb{P}_*$-a.s.
\begin{equation}
\dim_H \pi\Phi\mu \geq \beta(\pi_0)  \quad  \text{ for all } \pi\in \Pi= \pi_0 G.\label{ineqbeta}
\end{equation}
\end{corollary}
\begin{proof}
Since $E$ is lower semi-continuous, it follows from Theorem \ref{e}(i)  that for any $\epsilon>0$ the set
\[
\mathcal{U}_\epsilon=\{\pi\in \Pi_{d,k}: E(\pi) >\beta(\pi_0)-\epsilon  \}
\]
is open and dense in $\Pi$. Write $\mathcal{U}_\epsilon g=\{\pi g: \pi\in \mathcal{U}_\epsilon\}$ for $g\in {G}$. Then from Theorem \ref{e}(iii) we have $\mathbb{P}_*$-a.s. for ${\xi}$-a.e. $g\in {G}$,
\[
\widetilde{\mathcal{U}}_\epsilon=\{\pi\in \Pi: \dim_H\pi \Phi\bar{\mu} >\beta(\pi_0)-\epsilon  \}\supseteq \mathcal{U}_\epsilon g.
\]
Since $\mathcal{U}_\epsilon$ has non-empty interior, we deduce that $\mathbb{P}_*$-a.s. $\widetilde{\mathcal{U}}_\epsilon=\Pi$ as required. \end{proof}

\begin{corollary}\label{c2}
If $G=SO(d,\mathbb{R})$, then $\mathbb{P}_*$-a.s.
\begin{equation}
\dim_H \pi\Phi\mu = \min(k, \dim_H\Phi\mu)  \quad  \text{ for all } \pi\in \Pi_{d,k}.\label{ineqbeta}
\end{equation}
Moreover, with  $\alpha$ and $\beta(\pi)$ as in Theorem \ref{thmed}(i),(ii),  $\beta(\pi)=\min(k, \alpha)$ for all $\pi\in \Pi_{d,k}.$
\end{corollary}
\begin{proof}
If $G=SO(d,\mathbb{R})$, then $\pi_0 G=\Pi_{d,k}$ and for any $\pi\in \Pi_{d,k}$ there exists $g\in G$ such that $\pi_0g=\pi$. Due to the invariance of Haar meausres this implies that $\beta(\pi)=\beta(\pi_0)$ for all $\pi\in \Pi_{d,k}$, thus a constant. Then by Corollory \ref{c1} we get that $\mathbb{P}_*$-a.s. $\dim_H \pi\Phi\mu\geq \beta(\pi_0)$ for all  $\pi\in \Pi_{d,k}$, with equality for almost all $\pi$  by Theorem \ref{thmed}(ii). 
From the definition of dimension of measures (\ref{dimmes}), and applying the projection theorems of Marstrand \cite{Mar54} and Mattila \cite{Mat75} to sets $E$ with $\Phi\bar{\mu}(E) >0$ and $\dim_H E > \dim_H \Phi\bar{\mu} - \epsilon$, for $\epsilon >0$, it follows that $\mathbb{P}_*$-a.s. $\dim_H \pi\Phi\mu \leq \min(k, \dim_H \Phi\mu)=\min(k,\alpha)$ for all $ \pi\in \Pi_{d,k}$ with equality for a.a. $ \pi\in \Pi_{d,k}$. The conclusions follow.
\end{proof}

As in \cite{HoSh12} results on projections may be generalized to 
$C^1$-maps  without singular points, that is $C^1$-maps for which the derivative matrix is everywhere non-singular.
\begin{proposition}\label{pro}
Let $\pi\in \Pi=\pi_0 G$. For all $C^1$-maps $h:B(0,R)\mapsto \mathbb{R}^k$ such that $\sup_{x\in K}\|D_xh -\pi\|<c\rho^q$, we have that $\mathbb{P}_*$-a.s. for ${\xi}$-a.e. $g \in G$,
\[
\dim_H h g\Phi\bar{\mu}\ge E_q(\pi)-O(1/q),
\]
where the constant in $O(1/q)$ only depends on $\rho$, $c$, $R$ and $k$.
\end{proposition}

\begin{proof}
The proof is similar to that of \cite[Proposition 8.4]{HoSh12}.
\end{proof}

\begin{corollary}\label{non}
If $G=SO(d,\mathbb{R})$, then $\mathbb{P}_*$-a.s., for all $C^1$-maps $h:K\mapsto \mathbb{R}^k$ without singular points,
\begin{equation}
\dim_H h\Phi\mu=\min(k,\dim_H \Phi\mu).\label{eqnb}
\end{equation}
\end{corollary}

\begin{proof}
Corollary \ref{c2} together with Theorem \ref{e}(ii) yields that $E(\pi)=\min(k,\dim_H \Phi\mu)=\min(k,\alpha)$ is a constant for all $\pi\in\Pi_{d,k}$, and it is the maximum possible value since $h$ is a $C^1$ map. The result follows from Proposition \ref{pro}. 
\end{proof}

\section{Applications to self-similar sets}
\setcounter{theorem}{0}

Random cascade measures include non-random measures as a special case, so we can apply our results to the fractal geometry of deterministic self-similar sets. In this section we consider an IFS  $ \mathcal{I}$ of similarities on $\mathbb{R}^d$  (\ref{IFS}) with rotation group $G=\overline{\langle O_i:i\in \Lambda\rangle}= SO(d,\mathbb{R})$, with {\it self-similar attractor} $K$ the unique non-empty compact subset  of $\mathbb{R}^d$ satisfying 
$K = \cup_{i=1}^{m}f_i(K)$. 
Recall that $\mathcal{I}$ satisfies the \textit{strong separation condition} (SSC) if this union is disjoint and satisfies the  \textit{open set condition} (OSC) if there is a non-empty open set $V$ such that 
$V \subseteq \cup_{i=1}^{m}f_i(V)$ with this union disjoint. If either SSC or OSC are satisfied then
\begin{equation}
 \dim_H K = s \quad \mbox{ where } \sum_{i=1}^m r_i^s =1. \label{ssdim} 
\end{equation}

To  transfer our results to sets we need to ensure that the sets support suitable measures. From the definitions, if a probability measure $\nu$ is supported by a compact set $K$   then $\dim_H \nu\leq \dim_H K$. We say that an IFS  $\mathcal{I}$ with self-similar attractor $K$ satisfies the \textit{strong variational principle} if 
there is a Bernoulli probability measure $\mu$  on $\Lambda^{\mathbb{N}}$ such that 
$\dim_H \Phi\mu=\dim_H K$. No self-similar set with $G= SO(d,\mathbb{R})$ which does not satisfy the strong variational principle is known, and in particular the principle holds in the cases  described in the following lemma. 

\begin{lemma}\label{ifsmes}
(a)  If the IFS $ \mathcal{I}$ satisfies the open set (or strong separation) condition then $ \mathcal{I}$ satisfies the strong variational principle.

\noindent (b) Given  $0<r_i<\frac{1}{2}$ and $O_i$, the IFS $ \mathcal{I}$ in   (\ref{IFS}) satisfies the strong variational principle for  almost all $(t_1,\ldots,t_m)$ in the sense of $md$-dimensional Lebesgue measure.

\end{lemma}

\begin{proof}
(a) With $s$ given by (\ref{ssdim}), the Bernoulli probability measure $\mu$  on $\Lambda^{\mathbb{N}}$, defined by 
\begin{equation}
\mu^{[\emptyset]}([i]) = r_i^s \quad (i=1,\ldots,m), \label{bermes} 
\end{equation}
has  $\dim_H \Phi\mu=\dim_H K$. This fact is the key step in showing that $\dim_H K = s$ when OSC holds, see for example \cite{Hut81}.

(b) This follows by applying to self-similar sets the argument used in \cite{Fa86} to find the almost sure dimension of self-affine sets. With $\mu$ as in (\ref{bermes}), integrating the $t$-energy of the image measures $\Phi \mu$ over a parameterized family of self-similar sets  gives that the energy is bounded for almost all $(t_1,\ldots,t_m)$ for all  $t<s$, so that $\dim_H K = s$ for almost all $(t_1,\ldots,t_m)$. 
\end{proof}

The following two corollaries, obtained by applying Corollaries \ref{c2} and  \ref{non} to self-similar sets, weaken the conditions that guarantee the dimensions of projections and images from those of  \cite{HoSh12} to just  the strong variational principle.

\begin{corollary}\label{projss}
Let $K$ be the self-similar attractor of an IFS $ \mathcal{I}$ with rotation group $ SO(d,\mathbb{R})$  such that the strong variational principle is satisfied. Then
\[
\dim_H \pi K=\min(k,\dim_H K) \text{ for all } \pi\in \Pi_{d,k}.
\]
 \end{corollary}

\begin{corollary}\label{mapsss}
Let $K$ be the self-similar attractor of an IFS $ \mathcal{I}$  with rotation group $SO(d,\mathbb{R})$  such that the strong variational principle is satisfied. Then for all $C^1$-maps $h:K\to \mathbb{R}^k$ without singular points
\[
\dim_H h( K) =\min(k,\dim_H K).
\]
 \end{corollary}

The {\it distance set} of $A \subseteq \mathbb{R}^d$ is defined as $D(A) = \{|x-y|: x,y \in A\}$ and the {\it pinned distance set} of $A$ at $a$ is  $D_a(A) = \{|x-a|: x \in A\}$. A general open problem is to relate the Hausdorff dimensions and Lebesgue measures of $D(A)$ and $D_a(A)$ to that of $A$.
For self-similar sets in the plane, Orponen  \cite{Or12} showed that if $\dim_H K> 1$ then $ \dim_H D(K) = 1$. We have the following variant.

\begin{corollary}\label{distA}
Let $K$ be the self-similar attractor of an IFS $ \mathcal{I}$ with rotation group $SO(d,\mathbb{R})$  such that the strong variational principle is satisfied. Then there exists $a \in K$ such that
\[
\min(1,\dim_H K)=\dim_H D_a(K)\leq \dim_H D(K) \leq 1.
\]
\end{corollary}

\begin{proof}
Take a point $a \in K$, and some $i \in \Lambda$ such that $a \notin f_i(K)$. Then  $f_i(K)$ is similar to $K$, so by scaling, Corollary \ref{mapsss} applies to $C^1$-maps $h:f_i(K)\to \mathbb{R}^k$. The mapping  $h: f_i(K) \to \mathbb{R}$ given by $h(x) = |x-a|$ is $C^1$ and  has no singular points, so applying Corollary \ref{mapsss} to $f_i(K)$ gives 
$$ \dim_H  \{|x-a|: x \in f_i(K)\} = \dim_H  \{h(f_i(K))\}= \min(1, \dim_H f_i(K)) = \min(1, \dim_H K)$$
since  $f_i(K)$ is similar to $K$. Since $a \in K$ and $f_i(K) \subseteq K$, $ \{|x-a|: x \in f_i(K)\} \subseteq D_a(K)$.
\end{proof}

Furstenberg \cite{Fur} showed that if  a self-similar set has finite rotation group finite and satisfies the SSC then all directions are dimension conserving. Here we can dispense with the separation condition.

\begin{corollary}\label{dimcon}
Let $K$ be the self-similar attractor of an IFS $ \mathcal{I}$ with finite rotation group  such that the strong variational principle is satisfied. Then every direction is dimension conserving, that is 
for all $\pi\in\Pi_{d,k}$ there is a number $\Delta >0$ such that 
\begin{equation}
\Delta + \dim_H \{y \in \mathbb{R}^{k}: \dim_H (K\cap \pi^{-1}y) \geq \Delta\} \geq \dim_H K \label{dimcon} 
\end{equation}
(we take $\dim \emptyset = - \infty$).
\end{corollary}

\begin{proof}
This follows from Corollary \ref{dim-con2} taking $\Delta =\dim_H \Phi\bar{\mu}_{y,\pi}$ for some measure $\mu$ satisfying the strong variational principle.
\end{proof}

Examples such as the Sierpi\'{n}ski triangle \cite{Ken97} and the Sierpi\'{n}ski carpet \cite{ManSim13}  show that the value of $\Delta$ in (\ref{dimcon}) can vary with $\pi$.
\section{The percolation model}\label{percol}
\setcounter{theorem}{0}

Whilst fractal percolation or Mandelbrot percolation is most often based on a decomposition of a $d$-dimensional cube into $m^d$ equal subcubes of sides $m^{-1}$,  random subsets of any self-similar set may be constructed using a similar percolation process. Let $ \mathcal{I}=\{f_i=r_iO_i\cdot+t_i\}_{i=1}^{m}$ be an IFS of similarities with attractor $K$ and let $\mathbb{P}$ be a probability distribution on ${\mathcal  P}(\Lambda)$, the collection of all subsets of $\Lambda = \{1,\ldots,m\}$. We define a sequence of random subsets of $ \Lambda^n$ inductively as follows. The random set $S_1\subseteq \Lambda$ has distribution $\mathbb{P}$. Then, given $S_n$,  let  $S_{n+1}= \cup_{i \in S_n} S^i$ where $S^i =\{ ij : j \in S^i_1\}\subseteq \Lambda^{n+1}$ and  where $S^i_1\subseteq  \Lambda$ has the distribution $\mathbb{P}$  independently for each $i \in S_n$.
A  sequence of random subsets $\{K_n \}_{n=1}^{\infty}$ of $K$ is given by 
$K_n= \cup_{i \in S_n} f_{i}(K)$.  We write
$K_\mathbb{P} = \cap_{n=0}^\infty K_n$ for the resulting random compact subset of $K$ which is known as the \textit{percolation set}. (Note that standard Mandelbrot percolation on a cubic grid is a particular case of percolation on a self-similar set satisfying OSC.)

In this random setting we say that $(\mathcal{I},\mathbb{P})$ satisfies the {\it strong variational principle} if there exists a random cascade measure $\mu$  on $\Lambda^\mathbb{N}$ such that there is a positive probability  of $K_\mathbb{P} \not= \emptyset$, and such that, conditional on $K_\mathbb{P} \not= \emptyset$,
\begin{equation}\label{svp2}
\dim_H K_\mathbb{P}=\dim_H \Phi\mu=\alpha
\end{equation}
a.s., where $\alpha$ is given by Theorem \ref{thmed}(i). The next lemma gives a condition for $(\mathcal{I},\mathbb{P})$ to satisfy the strong variational principle, in which case $\alpha$ is given by an expectation equation.
\begin{lemma}
 Let $(\mathcal{I},\mathbb{P})$ be as above with $\mathcal{I}$ satisfying OSC and  with $\mathbb{E} \{\mathrm{card} S_1\}   > 1$. Then $(\mathcal{I},\mathbb{P})$ satisfies the strong variational principle with $\alpha$ given by
 \begin{equation}\label{percdim}
\mathbb{E}(\sum_{i\in S_1}  r_i^\alpha) = 1.
\end{equation}
\end{lemma}
\begin{proof}
By standard branching process theory \cite{AN}, if $\mathbb{E} \{\mbox{card} S_1\}   > 1$ there is a positive probability that  $K_\mathbb{P} \not= \emptyset$. Under OSC, conditional on  
$K_\mathbb{P} \not= \emptyset$ the a.s. dimension of  
$K_\mathbb{P} $ is  the solution $\alpha$ of
(\ref{percdim})
The random cascade defined by the random vector
\begin{equation}\label{wdef}
W= (W_1,\ldots,W_n) = (r_1^\alpha\chi_{\{1\in S\}}(\omega),\ldots, r_m^\alpha\chi_{\{m\in S\}}(\omega)). 
\end{equation} gives rise to a random measure $\Phi\mu$ supported by $K_\mathbb{P}$ such that $\mathbb{P}^*(K_\mathbb{P} \not= \emptyset)>0$.
Using a potential-theoretic estimate or a direct verification of the formula in Theorem \ref{thmed}(i), 
$\dim_H \Phi\mu = \dim_H K_\mathbb{P}= \alpha$ a.s., see \cite{Fa86,MW86}, so the $\alpha$ given by (\ref{percdim}) equals that of Theorem \ref{thmed}(i).
\end{proof}

Investigation of the dimensions of projections of the basic $m$-adic square-based percolation process goes back some years, see \cite{Dek} for a survey, and recently Rams and Simon \cite{RS}  showed using direct geometric arguments that a.s.  all orthogonal projections of square-based percolation have Hausdorff dimension $\min\{1,\alpha\}$, where $\alpha$ is the dimension of the  percolation set. The following application of Corollory \ref{c2} gives a similar conclusion for percolation on self-similar sets for which the IFS has dense rotations. 

\begin{corollary}\label{properc}
If $(\mathcal{I},\mathbb{P})$ satisfies the strong variational principle and  has rotation group $SO(d,\mathbb{R})$, then a.s. conditional on $K_\mathbb{P} \not= \emptyset$,
\[
\dim_H \pi K_\mathbb{P}=\min(k,\dim_H K_\mathbb{P})=\min(k,\alpha) \text{ for all } \pi\in \Pi_{d,k},
\]
where $\alpha$ is given by \eqref{svp2}.
\end{corollary}

Again, Corollary \ref{non} gives a variant for $C^1$-maps.

\begin{corollary}\label{nonperc}
If $(\mathcal{I},\mathbb{P})$ satisfies the strong variational principle and has  rotation group $SO(d,\mathbb{R})$, then a.s. conditional on $K_\mathbb{P} \not= \emptyset$,
\[
\dim_H h( K_\mathbb{P})=\min(k,\dim_H K_\mathbb{P}) =\min(k,\alpha)
\]
for all $C^1$-maps $h:K\to \mathbb{R}^k$ without singular points, where $\alpha$ is given by \eqref{svp2}.
\end{corollary}

Distance sets of percolation sets have also attracted interest recently, see \cite{RS} for the case of square-based percolation. The following result follows from a similar argument to that of Corollary \ref{distA} but in a random setting using Corollary \ref{nonperc}.

\begin{corollary}\label{dist}
Suppose that $(\mathcal{I},\mathbb{P})$ satisfies the strong variational principle and has  rotation group $ SO(d,\mathbb{R})$. Then a.s.  conditional on $K_\mathbb{P} \not= \emptyset$, there exists $a \in K_\mathbb{P}$ such that
\[
 \min(1,\alpha)=\dim_H D_a(K_\mathbb{P})\leq \dim_H D(K_\mathbb{P}) \leq 1,
\]
where $\alpha$ is given by \eqref{svp2}.
\end{corollary}

\section*{Acknowledgements}
The authors thank De-Jun Feng, Mike Hochman and Mike Todd for helpful discussions. We also thank the referees for many suggestions that have helped improve the paper.

\bibliographystyle{abbrv}

\begin{thebibliography}{10}

\bibitem{AN}
K.~B. Athreya and P.~E. Ney.
\newblock {\em Branching processes}.
\newblock Dover Publications Inc., Mineola, NY, 2004.

\bibitem{Ba99}
J. Barral.
\newblock Moments, continuit\'e, et analyse multifractale des martingales de Mandelbrot.
\newblock {\em Probab. Theory Related Fields}, 113: 535--569, 1999.

\bibitem{BaJi10}
J.~Barral and X.~Jin.
\newblock Multifractal analysis of complex random cascades.
\newblock {\em Comm. Math. Phys.}, 297: 129--168, 2010.


\bibitem{Dek}
M.~Dekking.
\newblock Random {C}antor sets and their projections.
\newblock In {\em Fractal Geometry and Stochastics {IV}}, volume~61 of {\em
  Progr. Probab.}, pp. 269--284. Birkh\"auser Verlag, Basel, 2009.

\bibitem{DuLi83}
R,~Durrett and T.~Liggett.
\newblock Fixed points of the smoothing transformation.
\newblock {\em Z. Wahrsch. Verw. Gebiete}, 64: 275--301, 1983.

\bibitem{Fa86}
K.~J. Falconer.
\newblock Random fractals.
\newblock {\em Math. Proc. Cambridge Philos. Soc.}, 100: 559--582, 1986.

\bibitem{Fal97}
K.~J. Falconer.
\newblock {\em Techniques in Fractal Geometry}.
\newblock John Wiley \& Sons Ltd., Chichester, 1997.

\bibitem{Falconer03}
K.~J. Falconer.
\newblock {\em Fractal Geometry -- Mathematical Foundations and Applications}.
\newblock John Wiley \& Sons Ltd., Chichester, 2nd Ed., 2003.

\bibitem{FalHow97}
K.~J. Falconer and J.~Howroyd.
\newblock Packing dimensions of projections and dimension profiles.
\newblock {\em Math. Proc. Cambridge Philos. Soc.}, 121: 269--286, 1997.

\bibitem{FeHu09}
D.-J.~Feng and H.~Hu.
\newblock Dimension theory of iterated function systems.
\newblock {\em Comm. Pure Appl. Math.}, 62: 1435--1500, 2009.

\bibitem{Fur}
H.~Furstenberg.
\newblock Ergodic fractal measures and dimension conservation.
\newblock {\em Ergod. Th. \& Dynam. Sys.}, 28: 405-422, 2008.

\bibitem{Ho12}
M.~Hochman.
\newblock Dynamics on fractals and fractal distributions, ar{X}iv:1008.3731v2, 2013.


 \bibitem{HoSh12}
M.~Hochman and P.~Shmerkin.
\newblock Local entropy averages and projections of fractal measures.
\newblock {\em Ann. of Math.(2)}, 175: 1001--1059, 2012.

\bibitem{Hut81}
J.~Hutchinson.
\newblock Fractals and self-similarity.
\newblock {\em Indiana Univ. Math. J.}, 30: 713-747, 1981.

\bibitem{KaPe76}
J.-P.~Kahane and J.~Peyri\'ere.
\newblock Sur certaines martingales de Benoit Mandelbrot.
\newblock {\em Adv. Math.}, 22: 131--145, 1976.

\bibitem{Kau68}
R.~Kaufman.
\newblock On Hausdorff dimension of projections.
\newblock {\em Mathematika}, 15: 153-155, 1968.

\bibitem{Ken97}
R.~Kenyon.
\newblock Projecting the one-dimensional Sierpinski gasket.
\newblock {\em Israel J. Math.}, 97: 221-238, 1997.

\bibitem{KeyNew76}
H.~B. Keynes and D.~Newton
\newblock Ergodic measures for nonabelian compact group extensions.
\newblock {\em Compositio Math.}, 32: 53-70, 1976.

\bibitem{Maker40}
P.~T. Maker.
\newblock The ergodic theorem for a sequence of functions.
\newblock {\em Duke Math. J.}, 6:27--30, 1940.

\bibitem{ManSim13}
A.~Manning and K.~Simon.
\newblock Dimension of slices through the Sierpinski carpet.
\newblock {\em Trans. Amer. Math. Soc.}, 365: 213--250, 2013.

\bibitem{Mar54}
J.~M. Marstrand.
\newblock Some fundamental geometrical properties of plane sets of fractional
  dimensions.
\newblock {\em Proc. London Math. Soc.(3)}, 4: 257--302, 1954.

\bibitem{Mar54a}
J.~M. Marstrand.
\newblock The dimension of Cartesian product sets.
\newblock {\em Proc. Cambridge Philos. Soc.}, 50:198--202, 1954.

\bibitem{Mat75}
P.~Mattila.
\newblock Hausdorff dimension, orthogonal projections and intersections with
  planes.
\newblock {\em Ann. Acad. Sci. Fenn. Ser. A I Math.}, 1: 227--244, 1975.

\bibitem{Mat95}
P.~Mattila.
\newblock {\em Geometry of sets and measures in Euclidean spaces, Fractals and rectifiability}.
\newblock Cambridge Studies in Advanced Mathematics 44.
\newblock Cambridge University Press, Cambridge, 1995.

\bibitem{MW86}
R.~D. Mauldin and S.~C. Williams.
\newblock Random recursive constructions: asymptotic geometric and topological
  properties.
\newblock {\em Trans. Amer. Math. Soc.}, 295: 325--346, 1986.

\bibitem{Or12}
T.~Orponen.
\newblock On the distance sets of self-similar sets.
\newblock {\em Nonlinearity}, 25: 1919--1929, 2012.

\bibitem{Pa97}
W.~Parry.
\newblock Skew products of shifts with a compact {L}ie group.
\newblock {\em J. London Math. Soc.(2)}, 56: 395--404, 1997.

\bibitem{PerSch00}
Y.~Peres and W.~Schlag.
\newblock Smoothness of projections, Bernoulli convolutions, and the dimension of exceptions.
\newblock {\em Duke Math. J.}, 102: 193--251, 2000.

\bibitem{RV}
R.~Rhodes and V.~Vargas.
\newblock Gaussian multiplicative chaos and applications: a review, ar{X}iv:1305.6221, 2013.

\bibitem{Ro49}
V.~A. Rohlin.
\newblock On the fundamental ideas of measure theory.
\newblock {\em Mat. Sbornik N.S.}, 25(67): 107--150, 1949.

\bibitem{RS}
M.~Rams and K.~Simon.
\newblock The dimension of projections of fractal percolations.
\newblock {\em J. Stat. Phys.}, 154: 633--655, 2014.

\end{thebibliography}

\end{document}